\documentclass[12pt]{article}
\usepackage{verbatim}
\usepackage{amssymb,amsmath}
\usepackage{amsthm}
\usepackage{graphicx}
\usepackage{epsfig}
\newtheorem{corollary}{Corollary}[section]
\newtheorem{lemma}[corollary]{Lemma}
\newtheorem{proposition}[corollary]{Proposition}
\newtheorem{theorem}[corollary]{Theorem}
\newcommand{\Prob} {{\mathbb P}}
\newcommand{\Z}{{\mathbb Z}} 
\newcommand{\E}{{\mathbb E}}

\newcommand{\R}{{\mathbb{R}}}
\newcommand{\C}{{\mathbb C}}

\newcommand{\hdim}{{\rm dim}_h}

\newcommand{\dist}{{\rm dist}}

\def \Im {{\rm Im}}
\def \Re {{\rm Re}}
\def \p {\partial}

\def \Half {{\mathbb H}}

\def \energy {{\cal E}}
\def \hcap {{\rm hcap}}
\def \rect{{\cal R}}
\def \distb{{\Psi}}
\def \distsub{{\Upsilon}}

\def \F {{\cal F}}

\newenvironment{remark}[1][Remark]{\begin{trivlist}
\item[\hskip \labelsep {\bfseries #1}]}{\end{trivlist}}
\newenvironment{definition}[1][Definition]{\begin{trivlist}
\item[\hskip \labelsep {\bfseries #1}]}{\end{trivlist}}
\newenvironment{example}[1][Example]{\begin{trivlist}
\item[\hskip \labelsep {\bfseries #1}]}{\end{trivlist}}

\begin{document}

\title{Dimension and natural parametrization for $SLE$
 curves}

\author{Gregory F. Lawler\\Department of Mathematics\\
University of Chicago \thanks{Research supported by National
Science Foundation grant DMS-0734151.}}

\maketitle
 
\begin{abstract}  Some possible definitions
for the natural parametrization of $SLE$ paths are proposed in terms of
various limits.    One of the definitions is used
to give a 
new proof that the Hausdorff dimension of $SLE_\kappa$ paths 
is $1 + \frac \kappa 8$ for $ \kappa < 8$.
\end{abstract}

\section{Introduction }

A number of measures on paths or clusters on two-dimensional
lattices arising from critical statistical mechanical
models are believed to exhibit some kind of conformal 
invariance in the
scaling limit.
The Schramm-Loewner evolution ($SLE$) was created
by Schramm \cite{Schramm} as a candidate for the scaling limit
of discrete measures on paths arising in statistical
physics.    $SLE$ is a continuous process which has the
conformal invariance built in --- in other words, it gives
possible (and in some cases, the only possible) candidates
for scaling limits assuming that these limits are conformally
invariant.

To give an example, let us consider the case of self-avoiding
walks (SAWs).  We will not be very precise here; in fact, what we say
here about SAWs is still only conjectural. Let $D \subset \C$ be
a bounded domain, which for ease we will assume has a smooth
boundary,  and let $z,w$ be distinct points on $\p D$. 
Suppose that a lattice $\epsilon \Z^2$ is placed on $D$ and let
$\tilde z, \tilde w \in D$ be lattice points in $\epsilon
\Z^2$ ``closest'' to $z,w$.  A self-avoiding walk (SAW) $\omega$
from
$\tilde z$ to $\tilde w$ is a sequence of distinct points
\[          \tilde z = \omega_0,\omega_1,\ldots,\omega_k=
    \tilde w , \]
with $\omega_j \in \epsilon \Z^2 \cap D$ and $|\omega_j
- \omega_{j-1}| = \epsilon$ for $1 \leq j \leq k$.  We write
$|\omega| = k$.
For each $\beta > 0$, we can consider the measure on SAWs from
$\tilde z$ to $\tilde w$ in $D$ that gives measure $e^{-\beta |\omega|},$
to each such
SAW.  There is a critical $\beta_0$, such that the
partition function
\[                \sum_{\omega: \tilde z \rightarrow
 \tilde w, \omega \subset D}  e^{-\beta_0 |\omega|} \]
neither grows nor decays exponentially as a function of
$\epsilon$ as $\epsilon \rightarrow 0.$   It is believed
that if we choose this $\beta_0$, and normalize so that this is
a probability measure, then there is a limiting measure on
paths that is the scaling limit.

It is believed that the typical length of a SAW in the measure
above is of order $\epsilon^{-d}$ where the exponent
$d = 4/3$ can be
considered the fractal dimension of the paths.  For fixed $\epsilon$,
let us define the scaled function
\[       \hat \omega ( {j} \epsilon^d )
                      =  \omega_{j} , \;\;\;\;
                           j=0,1,\ldots,|\omega|. \]
We can use linear interpolation to make this a continuous
path $\hat \omega:[0,\epsilon^d|\omega|] \rightarrow \C$. Then
it is conjectured that the following is true:
\begin{itemize}
\item  As $\epsilon \rightarrow 0$, the above probability measure
on paths converges to a probability measure $\mu_D^\#(z,w)$ supported
on continuous curves $\gamma:[0,t_\gamma] \rightarrow \C$ with
$\gamma(0) = z, \gamma(t_\gamma) = w, \gamma(0,t_\gamma) \subset
D$.
\item  The probability measures  $\mu_D^\#(z,w)$ are conformally
invariant.  To be more precise, suppose $F: D \rightarrow D'$
is a conformal transformation that extends to $\p D$ at least
in neighborhoods of $z$ and $w$.  For each $\gamma$ in $D$
connecting $z$ and $w$ consider the paths $F\circ \gamma$
(we define this precisely below).  Then this gives a measure
we call $F \circ \mu_D^\#(z,w)$.  The conformal invariance assumption
is 
\[              F \circ \mu_D^\#(z,w) = \mu_{D'}^\#(F(z),F(w)). \]
 
\end{itemize}

Let us now define $F \circ \gamma$.  The path $F\circ \gamma$
will traverse the points $F(\gamma(t))$ in order; the only question
is how ``quickly'' does the curve traverse these points.  If we look
at how the scaling limit is defined, we can see that if $F(z)
= rz$ for some $r> 0$, then the lattice spacing $\epsilon$ on
$D$ corresponds to lattice space $r \epsilon$ on $F(D)$ and hence
we  would expect the time to traverse $r \gamma$
should be $r^d$ times the time to traverse $\gamma$.  Using this
as a guide locally, we say that
the amount of time needed to traverse $F(\gamma[t_1,t_2])$ is 
\begin{equation}  \label{dec5.1}
        \int_{t_1}^{t_2 }
    |F'(\gamma(s))|^d \, ds . 
\end{equation}
This   tells us how to parametrize $F \circ \gamma$ and we include
this as part of the definition of $F \circ \gamma$.   This is analogous to the 
known conformal invariance of Brownian motion in $\C$ where the 
time parametrization must be defined as in \eqref{dec5.1} with
$d=2$.

We could  weaken our conjecture and only consider $\gamma$
and $F \circ \gamma$ as being defined only up to reparametrization.
In this case, one can show that the only candidate for the scaling
limit is (chordal)
Schramm-Loewner evolution ($SLE_\kappa$) for some
$\kappa > 0$; we define it more precisely later
but we give the basics here.  If there is to be a family of
probability measures $\mu_{D}^\#(z,w)$ (now being considered
modulo reparametrization) for simply connected $D$,
then we only need to define $\mu_\Half^\#(0,\infty)$.  If we impose
one further ``domain Markov'' property (which must be satisfied by
a scaling limit of SAW), then there is a one-parameter family of
probability measures $\gamma:[0,\infty) \rightarrow \overline
\Half$ from which the measure must come.  In the case of the
self-avoiding walk, another property called the ``restriction
property'' tells use which particular choice in this family is
the scaling limit \cite{LSWrest,LSWsaw}, but other choices
arise as scaling limits of other models.

Schramm's construction starts by giving a different parametrization
to the curve $\gamma$ in terms of a capacity in $\Half$ from
infinity.  When the curve $\gamma$ is given this parametrization
and $g_t$ denotes a conformal map from the slit domain
$\Half \setminus \gamma[0,t]$ to $\Half$, then $g_t(z)$
as a function  satisfies
a differential equation which goes back to Loewner.  This
parametrization can also be defined on the discrete level, and it
has been shown that for some models (e.g., loop-erased random walk
\cite{LSWlerw}, harmonic explorer \cite{SSharm}), the discrete
model with capacity parametrization converges to $SLE$ with
capacity parametrization. 

In this paper we consider two closely related questions for
$SLE_\kappa$:
\begin{itemize}
\item  Given $\gamma(t)$ parametrized by capacity, can one
recover the ``natural parametrization''?
\item How do we compute (rigorously) the Hausdorff dimension of the path
$\gamma[0,t)$?
\end{itemize}
It is known that for $\kappa \geq 8$, the paths of $SLE_\kappa$
fill the plane.  Let us restrict our discussion to $\kappa < 8$.
For these $\kappa$, it is known  that the
Hausdorff dimenison of $\gamma[0,t]$ is given by
\[           d = d_\kappa = 1 + \frac{\kappa}8 . \]
In the case $\kappa = 8/3$, this was first proved by
Schramm, Werner, and the author (see \cite{LBook} and
references therein) using the relationship between $SLE_{8/3}$
and the outer boundary of planar Brownian motion.
The upper bound for general $\kappa$
was first established by Rohde and Schramm \cite{RS}
by calculating the expecation of a derivative; we give a
form of this argument in Section \ref{uppersec}.  The lower
bound is much harder to establish.  This was done by Beffara
\cite{Beffara} using very intricate estimates.  

Establishing the lower bound is closely related to the question
of finding the natural parametrization.  Suppose that we can find
\[     \tau:[0,\infty) \rightarrow [0,\infty) , \]
such that $\eta(t) = \gamma(\tau(t))$ gives $SLE_\kappa$
in the natural parametrization.  Then we expect that $\eta$
induces a $d$-dimensional measure on the path $\gamma$ in the
sense that for all $\alpha < d$  
\begin{equation} \label{frost}
        \int_0^1 \int_0^1 \frac{ds \, dt}
                    {|\eta(s) - \eta(t)|^\alpha} < \infty . 
\end{equation}
 Frostman's lemma (see Section \ref{dimsec}) tells us that
the above condition implies that the Hausdorff dimension of
the path is at least $d$. A weaker version of
\eqref{frost} is sufficient to establish
 the lower bound on Hausdorff dimension: it suffices
to   find a (perhaps random) subset $J$ of $[0,1]$ such that
\begin{equation}  \label{jul25.1}
  \int_0^1 1_J(t) \, dt > 0, \;\;\;\;\;
     \int_0^1 \int_0^1 \frac{1_{J}(s) \, 1_{J}(t) \,ds \, dt}
                    {|\eta(s) - \eta(t)|^\alpha} < \infty. 
\end{equation}

In the next section, we will discuss a number of candidates for
the natural parametrization. We conjecture that they are all
equivalent up to multiplicative constant.
They are all described in terms
of limits that are hard to establish.  Although we do not prove
the limits exists, we do prove a kind of tightness result
for one of the definitions that allows us to take a subsequential
limit and construct a Frostman measure.  This gives a new proof of
the lower bound for the Hausdorff dimension of $SLE_\kappa$ paths
that was first proved by Beffara.
In \cite{LShef} we will establish the existence of the
natural parametrization at least for a range of
$\kappa$ including $\kappa = 8/3$.   

The majority of this paper is the proof of the lower bound
for the Hausdorff dimension combined with the  
derivative estimates needed to establish a second moment
bound. 
To give an idea, we would like to compare our
construction of a Frostman measure to that in \cite{Beffara}.
 Let $\gamma$ be an $SLE_\kappa$ curve
with $\kappa = 2/a < 8$.
In \cite{Beffara}, the starting point is to show that for
fixed $z \in \Half$, as $\delta \rightarrow 0+$,
\begin{equation}  \label{onepoint}
   \Prob\{\dist[z,\gamma(0,\infty)] \leq \delta\}
          \asymp G(z) \, \delta^{2-d}, 
\end{equation}
and to
 define a measure on $\gamma(0,1]$ to be $\delta^{d-2}$
times area restricted to the set $\{z
 \in \Half : \dist[z,\gamma(0,1]] \leq \delta\}$. Here
$G(z)$ denotes the ``Green's function'' for chordal
$SLE_\kappa$ in $\Half$,
\[   G(y(x+i)) = y^{d-2} \, (x^2 + 1)^{\frac 12 - 2a}. \]
 The 
estimate \eqref{onepoint} can be used to show that 
$\E[|\mu_\delta| ] \asymp 1$ as $\delta \rightarrow 0+$.
Here $|\cdot|$ denotes total mass.  A much harder argument
establishes 
  a two-point estimate
\begin{equation} \label{twopoint}
 \Prob\{\dist[z,\gamma(0,\infty)] \leq \delta,
  \dist[w, \gamma(0,\infty)] \leq \delta \}
          \leq c \, \left(\frac{\delta}{|z-w|}\right)^{2-d}
  \, \delta^{2-d}, \;\;\;\;\;  |z-w| \geq \delta.
\end{equation}
Once \eqref{twopoint} is established,
standard arguments, see, e.g. \cite{LBook}, can be used
to construct a Frostman measure.  Unfortunately, \eqref{twopoint}
is not at all easy to prove.  The problem is that the estimates
used to prove \eqref{onepoint} work well for a fixed $z$ but do
not handle two points $z,w$ as well.

A somewhat different approach is taken in this paper.
  We sketch the idea
here.  Suppose that $g_t$ are the conformal maps of SLE with
driving function $V_t$ where $V_t$ is a standard Brownian
motion. Let $\hat f_t(z) = g_t^{-1}(z + V_t)$.
 The Frostman measure is defined (approximately)
 to be the limit of
the measures $\mu_n$ where $\mu_n$ gives measure
\[        n^{-d/2} \, |\hat f'_{\frac{j-1}n}(i/\sqrt n)|^d \]
to $\gamma[\frac{j-1}{n}, \frac jn]$.
This approach also requires giving a second moment bound as well as
  showing that the curve $\gamma$ is sufficiently
``spread out'' so that the limit measure satisfies
\eqref{frost}.  We, in fact, consider a submeasure by considering
only certain good times and establishing a bound as in
\eqref{jul25.1}.  This method avoids complications for $\kappa
\geq 4$ where $SLE$ curves have double points and near double
points.  Roughly speaking, we consider a subset $J$ of times
for which $|\eta(s) - \eta(t)| \approx |s-t|^{1/d}$ for
$s,t \in J$. 

The main hurdle in this paper, which is also  
important in \cite{LShef}, is to estimate expectations
of the form 
\[           \E\left[|\hat f_s'(z)|^d \, |\hat f_{s+t}'(w)|^d
  \right], \;\;\;\;  z,w \in \Half. \]
This expectation can be written as
\[   \E \left[ |h_{s+t}'(w)|^d \, |\tilde h_{s}'(z)|^d \right],\]
where $h_{s+t}, \tilde h_s$ are conformal maps coming from
the reverse Loewner flow (see Section \ref{nov15sec} for
definitions).  Moreover,
\[             h_{s+t}'(w) = h_t'(w) \, \tilde h_s'(Z_t(w)) , \]
where $h_t$ is independent of $\tilde h_s$ and
$Z_t(w) = h_t(w) - U_t$ where $U_t$ is the driving function
for the reverse flow.  Hence, the expectation can be
written as
\[  \E \left[|h_t'(w)|^d \, |\tilde h_s'(Z_t(w))|^d \,
               |\tilde h_s'(z)|^d \right]. \]
Although the maps $h_t$ and $\tilde h_s$ are independent, the
random variable $ |\tilde h_s'(Z_t(w))|$ is not independent
of $|h_t'(w)|$ since $Z_t(w)$ depends on $h_t$.  In order
to estimate this expectation one needs to consider the
distribution of $Z_t(w)$ when one weights paths proportionally
to $|h_t'(w)|^d$.  This is the main topic of Section \ref{mainsec}
where the Girsanov theorem is combined with a particular martingale
to study this quantity.

Here is the basic outline of the paper.  The general discussion
of natural parametrization is in Section \ref{paramsec}.  The
next section concerns preliminary results: the reverse-time
Loewner flow, a general result about computing Hausdorff
dimension of random curves, and a simple distortion result
for conformal maps.  The proof of the lower bound
for Hausdorff dimension is done in Section \ref{proofsec}.  Here,
the approximations to the
Frostman measure is defined precisely and it is shown how to
derive the correlation estimates needed to establish the
limit as a Frostman measure.  The results in this section
rely on one lemma about the reverse Loewner flow at one point.
This lemma, as well as a number of related results, are discussed
in Section \ref{mainsec}.  For completeness we finish with two
sections quickly redoing results from \cite{RS}.  We prove
the upper bound for the Hausdorff dimension and we show the existence
of the curve for $\kappa < 8$.  This latter result is used
in the proof of our main theorem in Section \ref{proofsec} (although
we could have avoided using it) so we decided to include
it here. 

I would like to thank Tom Alberts for his detailed comments on
an earlier draft of this paper.

\section{Natural parametrization}  \label{paramsec}

\subsection{Schramm-Loewner evolution ($SLE$)}

We start by giving a quick review of the definition
of the Schramm-Loewner evolution; 
see \cite{LBook}, especially
Chapters 6 and 7, for more details.
  We will discuss only  chordal $SLE$ in
this paper, and we will call it just $SLE$.

Suppose  that $\gamma:(0,\infty) \rightarrow
\Half =\{x+iy: y > 0\}$ 
is a non-crossing curve with $\gamma(0+) \in
\R$ and $\gamma(t) \rightarrow \infty$ as $t \rightarrow 
\infty$.  Let $H_t$ be the unbounded component of
$\Half \setminus \gamma(0,t]$.   Using the Riemann mapping
theorem, one    can see that there is a unique conformal
transformation
\[            g_t: H_t \longrightarrow \Half \]
satisfying $g_t(z) - z \rightarrow 0$ as $z \rightarrow \infty$.
It has an expansion at infinity
\[           g_t(z) = z + \frac{a(t)}{z} + O(|z|^{-2}). \]
The coefficient $a(t)$  equals
$\hcap(\gamma(0,t])$ where $\hcap(A)$ denotes
the $\Half$-capacity from infinity of a bounded set
$A$.  There are a number of ways of defining $\hcap$, e.g.,
\[             \hcap(A) = \lim_{y \rightarrow \infty}
   y\,  \E^{iy}[\Im(B_{\tau})], \]
where $B$ is a complex Brownian motion and $\tau = \inf\{t:
B_t \in \R \cup A\}$.

\begin{definition}
The {\em Schramm-Loewner evolution},  $SLE_\kappa$, (from
$0$ to infinity in $\Half$) is the
 random curve $\gamma(t)$ such that $g_t$
satisfies
\begin{equation}
\label{loewner}
         \dot g_t(z) = \frac{a}{g_t(z) - U_t} ,\;\;\;\;
   g_0(z) = z, 
\end{equation}
where $a=2/\kappa$ and $U_t = -B_t$ is a standard Brownian
motion.  
\end{definition}

Showing that the 
conformal maps $g_t$ are well defined is easy.  It is not as obvious
that the curve $\gamma$ exists.  In Section \ref{curvesec} we
give a version of the argument first given by Rohde and Schramm
\cite{RS} that the curve exists for $\kappa \neq 8$.  
The argument actually shows more: there is a $\theta =
\theta(\kappa) > 0$ such that with probability one for
all $0 < T_1 < T_2 < \infty, 0 < y \leq 1$ there is a (random) $C$ such
that for that for all $T_1 \leq s < t \leq T_2$,
\begin{equation}  \label{jan8.3}
   |\gamma(s) - \hat f_s(y i)| \leq C y^{\theta} , \;\;\;\;
     |\gamma(s) - \gamma(t)|  \leq C \, (t-s)^\theta. 
\end{equation}
Although the argument here does not differ substantially from
that in \cite{RS}, we give it because we will use this result
and it follows quickly from the derivative exponents that
we derive.

\begin{remark}
We have defined chordal $SLE_\kappa$ so that it is {\em
parametrized
by capacity} with \[ \hcap(\gamma(0,t]) = at.\]  It is more often
defined with the capacity parametrization chosen so that
$ \hcap(\gamma[0,t]) = 2t$.  In this case we need to choose
$U_t = - \sqrt \kappa \, B_t$.  We will choose the parametrization
in (\ref{loewner}), but this is only for our convenience.
 Under our parametrization, if $z \in \overline \Half
\setminus \{0\}$, then 
 $Z_t = Z_{t}(z)  := g_t(z) - U_t$ satisfies the Bessel equation
\[          dZ_{t} = \frac{a}{Z_{t}}\, dt + dB_t. \]\end{remark}

We let 
\[  f_t = g_t^{-1},\;\;\;\; \hat f_t(z) = f_t(z + U_t). \] 
Define $g_{t,s}$ by $g_{t+s} = g_{t,s} \circ g_s$ and note
that for fixed $s$, $g_{t,s}$ is the solution of the 
initial value problem
\begin{equation}  \label{loewner2}
            \p_t g_{t,s}(z) = \frac{a}{g_{t,s}(z) - U_{s+t}},
\;\;\;\;  g_{0,s}(z) = z . 
\end{equation}
We also write $f_{t,s} = g_{t,s}^{-1} $ and note that
  \begin{equation}  \label{dec7.6}
 g_{t} = g_{t,0}, \;\;\;f_t = f_{t,0},\;\;\;
   f_{t+s} = f_s \circ f_{t,s}, \;\;\;f_{t,s} = g_s \circ f_{t+s}.
\end{equation}
We will make strong use of the following well known scaling
relation \cite[Proposition 6.5]{LBook}.
 
\begin {lemma}[Scaling]   \label{scalinglemma}
  If
$r > 0$, then the distribution of $g_{tr^2}(rz)/r$ is the same
as that of $g_t(z)$; in particular, $g_{tr^2}'(rz)$ has the
same distribution as $g'_t(z)$. 
\end{lemma}

  For
$\kappa < 8$,  we let
\begin{equation}  \label{nov28.20}
       d = 1 + \frac \kappa 8 = 1 + \frac 1{4a}. 
\end{equation}
We choose
 this notation because this is the Hausdorff dimension.
However, we do not assume this result in this paper, and 
so for now
this is only a choice of notation.

\subsection{Candidates}

Here we let $\kappa < 8$
and give a number of possible ways to obtain the natural
parametrization.  We expect that they all give the same
value up to multiplicative constant.  In each
case we will define approximate parametrizations
$\tau_n(t)$ and then the parametrization $\tau$
is given by   
\[
           \tau(t) = \lim_{n \rightarrow \infty}
             \tau_n(t). 
\]
We leave open the question of what kind of limit
can be taken here.   

Any candidate for the natural parametrization should satisfy
the appropriate scaling relationship.  In particular if
$\gamma(t)$ is an $SLE_\kappa$ curve, parametrized
so that $\hcap[\gamma(0,t]] = at$,  then $\tilde \gamma
(t) = r \gamma(t)$ is an $SLE_\kappa$ curve parametrized
so that  $\hcap[\gamma(0,t]] = r^2at$. If it takes time
$\tau(t)$ to traverse $\gamma(0,t]$ in the natural
parametrization, then it should take time
$r^d \, \tau(t)$ to traverse $\tilde \gamma(0,t]$ in
the natural parametrization.  In particular, it takes
time $O(R^d)$ in the natural parametrization to
go distance $R$.

\subsubsection{Minkowski measure}

Let  
\[    {\cal D}_{t,\epsilon}   =
              \{z \in \Half : \dist(z,\gamma(0,t]) \leq \epsilon\}, \]
 \[       \tau_n(t) =  n^{2-d}\, {\rm area}
\, (  {\cal D}_{t,1/n} )
   . \]
We call the limit, if it exists, the Minkowski measure of
$\gamma(0,t]$.  This terminology is somewhat misleading because
this is not a measure.
It is not too difficult to show that as $\epsilon \rightarrow
0+$, 
\begin{equation}  \label{jul25.2}
          \Prob\{z \in   {\cal D}_{\infty,\epsilon}
   \}  \asymp  G(z) \, \epsilon^{2-d} , 
\end{equation}
where, as before,
 $G(y(x+i)) = y^{d-2} (x^2+1)^{\frac{1-4a}{2}}$.   This shows
that $\tau_n(t)$ has been scaled so its expectation stays bounded
away from $0$ and infinity.  It is significantly more
difficult to derive the second moment bound,
\begin{equation}  \label{apr5.2}
  \Prob\{z, w\in   {\cal D}_{\infty,\epsilon}
   \}  \leq c \, G(z) \, G(w) \, \epsilon^{2(2-d)}
              \, |w-z|^{d-2}  . 
\end{equation}
With this bound, one can use this definition to prove the
lower bound on the Hausdorff dimension.  This is the strategy
used by Beffara and the hard work comes in proving
\eqref{apr5.2}.  Even with the second moment bound, it is 
not known how
to prove that the limit defining $\tau(t)$
exists.

There is a variant of the Minkowski measure that could be
called the conformal Minkowski measure.  Let $g_t$ be the
conformal maps as above.  If $t < T_z$,   let
\[               \Upsilon_t(z) = \frac{\Im[g_t(z)]}
              {|g_t'(z)|}. \]
A simple calculation shows that $\Upsilon_t(z)$ decreases in
$t$ and hence we can define
\[             \Upsilon_t(z) = \Upsilon_{T_z-}(z), \;\;\;\;
  t \geq T_z. \]
Similarly, $\Upsilon(z) = \Upsilon_\infty(z)$ is well defined.
The Koebe $1/4$-Theorem can be used to show that
$\Upsilon_t(z) \asymp \dist[z,\gamma(0,t] \cup \R]$; in fact, each
side is bounded above by four times the other side.  To prove
 \eqref{jul25.2} one can show that there is a $c_*$ such 
that
\[       \Prob\{\Upsilon (z) \leq \epsilon\}
               \sim c_* \, G(z) \, \epsilon^{2-d},
\;\;\;\;  \epsilon \rightarrow 0+. \]
This was first established in \cite{LBook} building on the
argument in \cite{RS}.  We give another proof of this in
Section \ref{uppersec}.  The conformal Minkowski measure
is defined as in the previous paragraph replacing $ {\cal D}_{t,\epsilon} $
with
\[ {\cal D}_{t,\epsilon}^* 
 =  \{z \in \Half :  \Upsilon_t(z)  \leq \epsilon\}. \]
It is possible that this limit will be easier to establish.
Assuming the limit exists, we can see that the expected amount of
time (using the natural
parametrization)
that $\gamma(0,\infty)$ spends in a bounded domain $D$
should be given (up to multiplicative constant) by
\begin{equation}  \label{jul25.4}
              \int_D G(z) \, dA(z) , 
\end{equation}
where $A$ dentoes area.  This observation is the starting point
for the construction of the natural parametrization in
\cite{LShef}.  However,
we will not use either version of the Minkowski measure
in this paper.

\subsubsection{$d$-variation}

Let
\[    \tau_n(t)  = \sum_{k=1}^{\lfloor tn
  \rfloor}  \left|\gamma\left(\frac kn \right)
   - \gamma \left(\frac{k-1}{n}\right)\right|^d . \]
More generally, we can consider
\[    \tau_n(t) = \sum_{t_{j-1,n} <t} \left|\gamma(t_{j,n})
   - \gamma(t_{j-1,n})\right|^d, \]
where $t_{0,n} < t_{1,n} < t_{2,n} < \infty$ is a partition,
depending on $n$, whose mesh goes to zero as $n \rightarrow \infty$.
One expects that for a wide class of partitions this limit
exists and is independent of the choice of partititions.
In the case $\kappa = 8/3$, a version of this
was studied numerically
by Kennedy \cite{Kennedy}. 

\subsubsection{Derivatives of inverse map}  \label{mysec}
We now describe the definition that will be used in
this paper.  We start with some heuristics.  Suppose
$\tau(t)$ were the natural parametrization. 
 Since $\tau(1) < \infty$, we would
expect that the average value of
\[               \Delta_n \tau (j) := \tau\left(
   \frac {j+1}n \right) - \tau\left(\frac{j}n\right) \]
would be of order $1/n$.  Consider
\[           \gamma^{(j/n)}\left[0,\frac 1n\right]
  = g_{j/n}\left(\gamma\left[\frac jn, \frac{j+1}n\right]
 \right). \]
Since the $\hcap$ of this set is $a/n$, we expect that the
diameter of the set is of order $1/\sqrt n$.  Using
the scaling properties, we guess that the time needed
to traverse $ \gamma^{(j/n)}\left[0,\frac 1n\right]$
in the natural parametrization is of order $n^{-d/2}.$
Using
the scaling properties again, we guess that 
\[            \Delta_n \tau (j) \approx 
                 n^{-d/2} \, |\hat f'_{j/n}(i/\sqrt n)|^d. \]
This leads us to define
\begin{equation}  \label{apr5.5}
           \tau_n(t) =  \sum_{k=1}^{\lfloor tn
  \rfloor}    n^{-d/2} \, |\hat f'_{k/n}(i/\sqrt n)|^d. 
\end{equation}
 This is the form we will use.

More generally, we could let
\[     \tau_n(t) =  \sum_{k=1}^{\lfloor tn
  \rfloor}    n^{-d/2} 
  \int_\Half  |\hat f'_{k/n}(z/\sqrt n)|^d\, \nu(dz), \]
where $\nu$ is a finite measure on $\Half$. 
This approach   used in \cite{LShef} to contruct the
natural parametrization starts with \eqref{jul25.4} but
ends up with a version of this for a particular measure
$\nu$.

\subsection{Lower bound on Hausdorff dimension}

In this paper, we will use the ideas from Section \ref{mysec}
to prove that for $\kappa <8$, the Hausdorff dimension 
of the paths is $d = 1 + \frac \kappa 8$.
 We will
focus on the proof of the lower bound which is the hard direction.
(For completeness sake we sketch a proof of the upper
bound in Section \ref{uppersec}.)  Since Hausdorff dimension is
preserved under conformal maps, it is easy to use the independence
of the
increments of   Brownian motion to conclude that there is a $d_*$
such that with probability one $ \hdim[\gamma[t_1,t_2]] = d_*$
for all $t_1 < t_2$.  Using this and the upper
bound, we can see that it suffices
to prove that for all $\alpha < d$, 
\[            \Prob\{  \hdim(\gamma[1,2])  \geq \alpha\}
    > 0. \]

  Since
$\E[\tau_n(t)] \asymp 1$, \eqref{apr5.5} suggests the
relation
\begin{equation}  \label{nov29.1.1}
  \E[|\hat f_1'(i/\sqrt n)|^d]  = 
  \E[|\hat f_n'(i)|^d] \asymp n^{\frac d2 - 1} .
\end{equation}\
(The first equality holds by scaling.)  
This was proved in \cite{RS}.  We give another proof here that  
derives additional information.
Let \[   \beta = d - \frac 32 = \frac{1-2a}{4a} = \frac{\kappa -4}{
  8}, \]
\[ \xi = d(d-2) +1
  =  \beta d + 1 - \frac{d}{2}  = \frac{1}{16a^2} 
   =\frac{\kappa^2}{64}  .\] 
 Note that $\beta < 1/2$ and $0 < \xi < 1$.
In our proof of \eqref{nov29.1.1}
 one sees that  
  $\E [| \hat f_t'(i)|^d]$ is not of the same order
of magnitude
as $\E[   |\hat f_t'(i)|]^d $ .  Roughly speaking,
the
expectation of $ |\hat f_t'(i)|^d$ is supported on the event that  
$|\hat f_t'(i)| \approx t^{\beta}$.  In order to
make this precise, it will be convenient to introduce
some terminology.

\begin{definition}
A function $\phi:[0,\infty) \rightarrow
(0,\infty)$ is a {\em subpower function}
if it is increasing, continuous, and
\[              \lim_{x \rightarrow \infty} \frac{\log \phi(x)}
   {\log x} = 0 , \]
i.e., $\phi$ grows slower than $x^q$ for all $q > 0$.
\end{definition}

The class of subpower functions is closed under addition
and multiplication.
In Theorem \ref{dec1.theorem2}, it is proved
that  there is a subpower function $\phi$ such
\[   \E[ |\hat f_t'(i)|^d]
  \asymp  \E\left[|\hat f_t'(i)|^d ; t^\beta \, \phi(t)^{-1} 
  \leq |\hat f_t'(i)| \leq t^\beta \, \phi(t)
   \right]  \asymp t^{\frac d2 - 1} , \]
which implies that
\begin{equation}  \label{apr5.6}
  \Prob\left\{ t^\beta \, \phi(t)^{-1} 
  \leq |\hat 
  f_t'(i)| \leq t^\beta \, \phi(t)\right\} \approx t^{\frac d2 -1
  - \beta d} = t^{-\xi}. 
\end{equation}

The basic idea underlying our construction of a Frostman measure
on the path is to replace $|\hat f'_{j/n}(i/\sqrt n)|^d$ in \eqref{apr5.5}
with $|\hat f'_{j/n}(i/\sqrt n)|^d \, 1_{E_{j,n}}$ where $E_{j,n}$
is an event, measurable with respect to $U_t, 0 \leq t \leq j/n$,
which roughly corresponds to (the scaled version of) the
event in \eqref{apr5.6}.

\section{Some preliminaries}

\subsection{Reverse time}  \label{nov15sec}

It is known \cite{RS,Kam} that estimates for
$\hat f_t'$ are often more easily derived by considering
the reverse (time) Loewner flow.     In this subsection,
we review the facts about
 the Loewner equation in reverse time that we will
need.  Suppose that $g_t$ is the solution to the
Loewner equation
\begin{equation} \label{loewner22}
   \p_t g_t(z) = \frac{a}{g_t(z) - V_t} , \;\;\;\;
        g_0(z) = z.
\end{equation}
Here  $V_t$ can be any continuous function, but we
will be interested in the case where $V_t$ is
a standard Brownian motion.

For fixed $T > 0$, let $ F_t^{(T)}, 0 \leq t \leq T$,
denote the solution to the time-reversed Loewner
equation
\begin{equation}  \label{dec7.5}
    \p_t F_t^{(T)}(z) = - \frac{a}{F_t^{(T)}(z) - V_{T-t}}
   =  \frac{a}{V_{T-t}- F_t^{(T)}(z)  },
\;\;\;\; F_0^{(T)}(z) = z . 
\end{equation}
Note that
\[   F_{s+T}^{(S+T)} (z
  ) = F_s^{(S)} ( F_{T}^{(S+T)}(z)), \;\;\;\;
  0 \leq s \leq S. \]

\begin{lemma}  \label{nov14.lemma1}
If $t \leq T$, then $F_{t}^{(T)} = f_{t,T-t}$.
In particular, $F_T^{(T)} = f_T$.
\end{lemma}

\begin{proof}
  Fix $ T$, and let
$u_t = F_{T-t}^{(T)}$.  Then \eqref{dec7.5} implies
that $u_t$ satisfies
\[  \dot u_t(z) = \frac{a}{u_t(z) - V_{t}}, \;\;\;\;
        u_{T}(z) = z . \]
By comparison with (\ref{loewner22}), we can see that
$u_{t}(z) = g_t(f_{T}(z))$, and we have
already noted in \eqref{dec7.6} that   $g_t \circ
f_T = f_{t,T-t}$.  
\end{proof}

We will be using the reverse-time flow, to study the behavior
of $\hat f$ at one or two times. We leave the simple
derivation of the next lemma from the previous
lemma to the reader.  A primary purpose of stating this lemma
now is to set the notation for future sections.

\begin{lemma}  \label{apr4.lemma2}
 Suppose $S,T > 0$ and $V:[0,S+T] \rightarrow
\R$ is a continuous function.  Suppose $g_t, 0 \leq t \leq
S+T$ is the solution to \eqref{loewner22}.
As before, let
  $f_t = g_t^{-1}$ and $\hat f(z) = f_t(z + V_t)$.
Let
\[           U_t = V_{S+T - t} - V_{S+T}, \;\;\;\;
                  0 \leq t\leq S+T, \]
\[  \tilde U_t = V_{S-t} - V_S = U_{T+t} - U_T, \;\;\;\;
   0 \leq t \leq S. \]
and let $h_t, 0 \leq t \leq S+T, \; \tilde h_t, 0 \leq t \leq
S$, be the solutions to the reverse-time Loewner equations
\[               \p_t h_t(z) = \frac{a}{U_t - h_t(z)},
\;\;\;\; h_0(z) = z , \]
\[  \p_t \tilde h_t(z) = \frac{a}{\tilde U_t - \tilde h_t(z)}, \;\;\;\;
   \tilde  h_0(z) = z . \]  
Then
\[   \hat f_S(z) =\tilde h_S(z) - \tilde U_S, \;\;\;\;
       \hat f_{S+T}(z) = h_{S+T}(z) - U_{S+T} , \]
\[     h_{S+T}(z) = \tilde h_S(h_T(z) - U_T) + U_T . \]
In particular,
\[  \hat f_S'(w) \, \hat f_{S+T}'(z) =
          h_T'(z) \, \tilde h_S'(h_T(z) - U_T) \,
           \tilde h_S'(w) . \]

\end{lemma}

If $S>T$, we will also need to consider $\hat f_{S-T}$.  We
will use $\hat h_t, \bar U_t, 0 \leq t \leq S-T$ for the 
corresponding quantities.

\begin{remark}  If $V_t$ is a Brownian motion starting at the
origin, then $U_t, \tilde U_t$ are standard Brownian motions
starting at the origin.  Moreover $\{U_t: 0 \leq t \leq T\}$
and $\{\tilde U_t: 0 \leq t \leq S\}$ are independent.
\end{remark}

\subsection{Hausdorff dimension}  \label{dimsec}

The main tool for proving lower bounds for Hausdorff
dimension is Frostman's lemma (see \cite[Theorem 4.13]{Falconer}),
 a version of which we recall
here: if $A \subset \R^m$ is compact
  and $\mu$ is a Borel measure supported on
$A$ with  
$0 <\mu(A) < \infty
  $ and 
\[     \energy_\alpha(\mu) := \int \int \frac{\mu(dx) \, \mu(dy)}
               {|x-y|^\alpha} < \infty, \]
then the Hausdorff-$\alpha$ measure of $A$
is infinite. In particular,
$\hdim(A) \geq \alpha$.   
 The next lemma is similar to many
that have appeared before (see \cite[A.3]{LBook}), but the exact formulation is what is
needed here.

\begin{lemma} \label{nov22.1}
Suppose $\eta:[0,1] \rightarrow \R^m$ is
a  random curve  and \[ \{F(j,n):n=1,2,\ldots, j=1,2,\ldots n\}\]
are nonnegative  random variables all defined on the same
probability space. Suppose that there exist $0 < \delta < \delta'< 1,
0 < C_1,C_2,C_3,C_4 < \infty$,  $ 0 <\alpha < m$  
  such that the following hold for $n=1,2,\ldots$ and
$1 \leq j \leq k \leq n$:
\begin{equation}  \label{nov29.6}
  C_1 \leq \frac 1n \sum_{j=1}^n \E[F(j,n)] \leq C_2  , 
\end{equation}
\begin{equation}  \label{nov29.7}
   \E[F(j,n) F(k,n)] \leq  C_3\, \left(  \frac{n}
  {k-j+1} \right)^{\delta} ,   
\end{equation}
and  
\begin{equation}  \label{nov29.8}
\left| \eta\left(\frac jn \right)
    - \eta \left( \frac kn \right) \right|^\alpha
\geq   C_4\,  \left(\frac{k-j}{n}\right)^{1-\delta'} \,
1\{F(j,n)F(k,n) > 0\}.  
\end{equation}
Then  
\[   \Prob \left\{\hdim(\eta[0,1]) \geq \alpha\right\} > 0  . \]
\end{lemma}

\begin{remark}
The  proof  constructs a measure supported on the
curve.  The $n$th approximation is a sum of
measures $\mu_{j,n}$ which are multiples of
Lebesgue measure on  small discs centered at $\eta(j/n)$.
The multiple at $\eta(j/n)$
is chosen so that the total mass $\mu_{j,n}$ 
is $F(j,n)/n$.  In particular, if
$F(j,n) = 0$,  $\mu_{j,n}$ is the zero
measure.  To apply Frostman's lemma,
we need to show that
the limiting measure is sufficiently spread out and \eqref{nov29.8}
gives the necessary assumption.  Note  the
assumption only requires the inequality
to hold when $F(j,n)F(k,n) > 0$.  The assumption
implies that if $j < k$ and $\eta(j/n) = \eta(k/n)$ (or
are very close), then at most one of $\mu_{j,n}$ and
$\mu_{k,n}$ is nonzero.  
\end{remark}

\begin{proof}  We fix $\epsilon > 0$ such 
that $\alpha  
 \leq m -  \epsilon$ and $\epsilon \leq \delta\leq \delta'
- \epsilon \leq 1-2\epsilon$.
Constants in this proof  depend on $m$ and
$\epsilon$.
Note that \eqref{nov29.7} and \eqref{nov29.8}
 combine to give
\begin{equation}  \label{nov29.5}
  \E\left[\frac{F(j,n) \, F(k,n)}{|\eta(j/n) - \eta(k/n)|^\alpha}
  \right]  \leq \frac{C_3}{C_4} 
\,  \left(\frac{n}{k-j}\right)^{1-(\delta'-\delta)},
\;\;\;\; j < k. 
\end{equation}
 Let 
$\mu_{j,n}$ denote the (random) measure
which is a multiple of
 Lebesgue measure on the disk of radius $n^{-(1-\delta)/\alpha}$
about $\eta(j/n)$ where the multiple
is chosen  so that $|\mu_{j,n}| = n^{-1} \,F(j,n)$.
Here $|\cdot|$ denotes total mass.
Let $\nu_n = \sum_{j=1}^n \mu_{j,n}$.   From \eqref{nov29.6}, we see
that
\[             \E[|\nu_n|] \geq C_1, \]
and from \eqref{nov29.7} we see that
\[  \E[|\nu_n|^2] = \frac 1{n^2}
\sum_{j=1}^n \sum_{k=1}^n
        \E[F(j,n) \, F(k,n)] \leq  \frac{C_3}{n^2}
              \sum_{j=1}^n \sum_{k=1}^n 
\left(  \frac{n}
  {|k-j|+1} \right)^{\delta} \leq c \, C_3   ,\]
   We will now show that 
\begin{equation}  \label{dec7.1}
    \E[\energy_\alpha(\nu_n)] =
\sum_{j=1}^n \sum_{k=1}^n \E\left[\int \int
      \frac{\mu_{j,n}(dx) \, \mu_{k,n}(dy)}{
   |x-y|^\alpha } \right]  \leq c \, \frac{C_3}{C_4 \wedge 1}, 
\end{equation} 
using the easy estimate 
\[  \int_{|x-x_0| \leq r}
 \int_{|y - y_0| \leq r} \frac{d^mx \, d^my}{|x-y|^{\alpha}}
    \leq c \, r^ {2m} \, \min\{r^{-\alpha}, |x_0 - y_0|^{-\alpha}\}. \]
To estimate the terms with $j=k$, note that  
   \eqref{nov29.7} gives
\[ \E\left[\int \int
      \frac{\mu_{j,n}(dx) \, \mu_{j,n}(dy)}{
   | x  -  y |^\alpha } \right] \leq c \, \E\left[\frac {F(j,n)^2 \,
 n^{1-\delta}} {n^2} \right] \leq c \,C_3 \, n^{-1}, \]
and hence
   \[  \sum_{j=1}^n \E\left[\int \int
      \frac{\mu_{j,n}(dx) \, \mu_{j,n}(dy)}{
   | x  -  y |^\alpha } \right] \leq c \, C_3. \]
  For
$j < k$, we use the estimate
\[    \int \int
      \frac{\mu_{j,n}(dx) \, \mu_{k,n}(dy)}{
   | x  -  y |^\alpha } 
    \leq c \,\frac{  F(j,n) \, F(k,n)}{n^2 \, |\eta(j/n) - \eta(k/n)|
   ^{ \alpha}}. \]
Combining this with \eqref{nov29.5} gives,
\[   \sum_{1 \leq j < k \leq n} \E\left[\int \int
      \frac{\mu_{j,n}(dx) \, \mu_{k,n}(dy)}{
   | x  -  y |^\alpha } \right]
  \leq  c\, \frac{C_3}{C_4\,n^{2}} \sum_{1 \leq j < k \leq n} 
   \left(\frac{n}{k-j}\right)^{1-(\delta'-\delta)} 
  \leq c \, \frac{C_3}{C_4}.\]
This gives \eqref{dec7.1}.
 
Standard arguments show that there is a $p > 0$,
such that
 with  probability  at least $p$,  
\[         |\nu_n|  \geq p, \;\;\;\;
    \energy_\alpha(\nu_n) \leq 1/p\mbox{ for infinitely many } n. \]
On this event we can  
can take a subsequential limit $\nu$ with $|\nu| \geq p$ and
$\energy_\alpha(\nu)  \leq 1/p $.  Since $\nu$
 must be supported on $\eta[0,1]$,  
   the conclusion
follows from Frostman's lemma.
\end{proof}

\begin{remark} In the proof we chose $\mu_{j,n}$ to a multiple
of Lebesgue measure on a small disk centered at $\eta(j/n)$
rather than choosing it
to be a point mass at $\eta(j/n)$.  We needed to do this in order to
establish \eqref{dec7.1}.  If we did not spread out the measure
a little bit, the terms in the sum with $j=k$ would be infinite.
\end{remark}

\begin{corollary}  \label{Frostmancorollary}
Suppose $\eta:[0,1] \rightarrow \R^m$ is
a  random curve  and \[ \{F(j,n):n=1,2,\ldots, j=1,2,\ldots n\}\]
are nonnegative  random variables all defined on the same
probability space.
Suppose $1 < d \leq m$,  and there exist a subpower
function $\psi$, $0 < \xi < 1$,    and  $c < \infty$ 
  such that the following holds for $n=1,2\ldots,$ and $1 \leq j \leq 
k \leq n$.
\begin{equation}  \label{nov15.h1}
 c^{-1} \leq \frac 1n \sum_{j=1}^n \E[F(j,n)] \leq c .
\end{equation}
\begin{equation}  \label{nov15.h2}
   \E[F(j,n) F(k,n)] \leq   \left(  \frac{n}
  {k-j+1} \right)^{\xi} \, \psi\left(\frac{n}
  {k-j+1}\right) , 
\end{equation}
and  
 \begin{equation}  \label{nov15.h3}
\left| \eta\left(\frac jn \right)
    - \eta \left( \frac kn \right) \right|^d
\geq  \left(\frac{k-j}{n}\right)^{1-\xi} \,
 \psi\left(\frac{n}
  {|j-k|+1}\right)^{-1}\, 
1\{F(j,n)F(k,n) > 0\}.  
\end{equation}
Then for each $\alpha < d$,
\[   \Prob\{\hdim(\eta[0,1]) \geq \alpha\} > 0  . \]
In particular, if it is known that there is a $d_*$ such that
$\Prob\{ \hdim(\eta[0,1]) = d_*\} = 1$, then $d_* \geq d$.
\end{corollary}

\begin{proof}  Suppose $\alpha < d$.  Let
\[   r =    \frac{d-\alpha}{d+\alpha} \, (1-\xi)  \]
By \eqref{nov15.h3}, there is a $c_1$ such that for $j < k$, if
$F(j,n) \, F(k,n) > 0$, 
\[ \left| \eta\left(\frac jn \right)
    - \eta \left( \frac kn \right) \right|^{-d}
\leq c_1 \,  \left(\frac n{k-j} \right)^{1-\xi+r} ,\]
and hence,
\[ \left| \eta\left(\frac jn \right)
    - \eta \left( \frac kn \right) \right|^{-\alpha}
\leq c_1 \,  \left(\frac n{k-j} \right)^{1-\xi -r} .\]
  By \eqref{nov15.h2}, there is a $c_2$
such that 
 \[ \E[F(j,n) F(k,n)] \leq  c_2  \left(  \frac{n}
  {k-j+1} \right)^{\xi + (r/2)}. \] 
The hypotheses of Lemma
\ref{nov22.1}  hold with  $\delta = \xi + (r/2),
\delta' = \xi + r$.\end{proof}

\begin{remark}
If the assumption \eqref{nov15.h2} is strengthened to
\[   \E[F(j,n) \, F(k,n)] \leq c\,
 \left(  \frac{n}
  {|j-k|+1} \right)^{\xi}, \]
then the conclusion can be strengthed to
\[            \Prob\{\hdim(\eta[0,1]) \geq d)\} > 0. \]
\end{remark}

\begin{remark}
One important case of this lemma is when $\eta$ is the
identity function and $d = 1-\xi$
in which case \eqref{nov15.h3} is immediate.
\end{remark}

It will be useful for us to give a slight generalization
of Corollary \ref{Frostmancorollary}. Corollary \ref{Frostmancorollary}
is the particular case of Corollary \ref{Frostmancorollary2} with
$\eta(j,n) = \eta(j/n)$.

\begin{corollary}  \label{Frostmancorollary2}
Suppose $\eta:[0,1] \rightarrow \R^m$ is
a  random curve, \[ \{F(j,n):n=1,2,\ldots, j=1,2,\ldots n\}\]
are nonnegative  random variables, and
\[  \{ \eta(j,n):n=1,2,\ldots, j=1,2,\ldots n\}\] 
are $\R^m$-valued random variables  all defined on the same
probability space. 
 Suppose there exist  $c < \infty$; 
 and a subpower
function $\psi$    
  such that the following holds
for $n=1,2\ldots,$ and $1 \leq j \leq 
k \leq n$.
\begin{equation}  \label{nov15.h1.A}
  \frac 1 c \leq \frac 1n \sum_{j=1}^n \E[F(j,n)] \leq c .
\end{equation}
\begin{equation}  \label{nov15.h2.A}
   \E[F(j,n) F(k,n)] \leq   \left(  \frac{n}
  {k-j+1} \right)^{\xi} \, \psi\left(\frac{n}
  {k-j+1}\right) , 
\end{equation}
 \begin{equation}  \label{nov15.h3.A}
\left| \eta(j,n) - \eta(k,n)\right|^d
\geq \frac 1c\, \left(\frac{|k-j|}{n}\right)^{1-\xi} \,
 \psi\left(\frac{n}
  {|j-k|+1}\right)^{-1}\, 
1\{F(j,n)F(k,n) > 0\}, 
\end{equation}
and such that with probability one
\begin{equation}  \label{nov15.h4.A}
  \lim_{n \rightarrow \infty}
        \max\left\{\dist[\eta(j,n),\eta[0,1]]:j=1,\ldots,n
  \right\} = 0 . 
\end{equation}
Then for each $\alpha < d$,
\[   \Prob\{\hdim(\eta[0,1]) \geq \alpha\} > 0  . \]
In particular, if it is known that there is a $d_*$ such that
$\Prob\{ \hdim(\eta[0,1]) = d_*\} = 1$, then $d_* \geq d$.
\end{corollary}

\begin{proof}  The proof proceeds as in Lemma \ref{nov22.1}
and Corollary \ref{Frostmancorollary}.  The measure $\mu_{j,n}$
in the proof of Lemma \ref{nov22.1} is placed on the disk
centered at $\eta(j,n)$ rather than $\eta(j/n)$.   The
key observation is that on the event \eqref{nov15.h4.A},
any subsequential limit of the measures $\nu_n$ must be
supported on $\eta[0,1]$.
\end{proof}

\subsection{A lemma about conformal maps}  \label{prelsec}

In this section, we state a result about conformal maps
that we will need.
For $r \geq 1$, let
\[ \rect(r)  = [-r,r] \times [1/r,r] =
 \{x+iy:  -r \leq x \leq r , \;\;\;\;
    1/r \leq y \leq r \}. \]
The proof of the next lemma is simple and left
to the reader.  It is an elementary way of stating
the fact that 
 the diameter of $\rect(r)$
in the  hyperbolic metric is $O(\log r)$.  
  We then use that lemma to prove
another lemma which is essentially a corollary of
the Koebe $(1/4)$-theorem and the distortion 
theorem (see \cite[Section 3.2]{LBook}) which we recall 
now.  Suppose $f:D \rightarrow \C$ is a conformal transformation,
$z \in D$, and $d = \dist(z, \p D)$.  Then, 
$f(D)$ contains the open
ball of radius $d\, |f'(z)|/4$ about $f(z)$ and
\begin{equation}  \label{distortion}
          \frac{1-r}{(1+r)^3} \, |f'(z)|
     \leq |f'(w)| \leq \frac{1 + r}{(1-r)^3}\, |f'(z)| \;\;\;\;
               |w-z| \leq rd .
\end{equation}  

\begin{lemma}
  There is a $c_2 < \infty$ such that for every $r \geq 1$
and  $z,w \in \rect (r)$,
we can find $z=z_0,z_1,\ldots,z_k = w$ with $k \leq c_2\, \log(r+1)$;
$z_0,\ldots,z_k \in \rect(r)$; and such that for $j=1,\ldots,k$, 
\[
|z_j - z_{j-1}| < \frac 18 \, \max\{\Im(z_j), \Im(z_{j-1})\} . \] 
\end{lemma}

\begin{lemma}  \label{nov29.lemma1}
There exists $\alpha< \infty $ such that if $g:
\Half \rightarrow \C$ is a conformal transformation, $r \geq 1$
and 
$        z,w \in  \rect(r),$
then  
\[          |g'(w)|  
    \leq   (2r)^{\alpha}\, |g'(z)|   ,  \]
\[    \dist[g(w), \p g(\Half)]
    \leq (2r)^{\alpha}\,
\dist[g(z),\p g(\Half)]  . \]
\end{lemma}

\begin{proof}  The Koebe-$(1/4)$ Theorem and
the Distortion Theorem imply  that there
is a $c$ such that 
if $z \in \Half$ and $|w-z| \leq \Im(z)/8$, then
\[ c^{-1} \, |g'(z)|  \leq |g'(w)| \leq c \, |g'(z)| \]
and
\[   c^{-1} \, {\dist[g(z),\p g(\Half)]}
   \leq \dist[g(w), \p g(\Half)] \leq c \, \dist[g(z), \p g(\Half)]. \]
  Hence if $z,w \in \rect(r)$ and $z_0,\ldots,z_k$ are as above,
\[  |g'(w)| \leq c^k \, |g'(z)| \]
and
\[    \dist[g(w), \p g(\Half)] \leq c^k \, \dist[g(z), \p g(\Half)]. \]
But, $c^k \leq c^{c_2 \, \log(r+1)} = (r+1)^\alpha  
 \leq(2r)^\alpha,$ where $\alpha = c_2 \log c$.
\end{proof}

\section{Proof}  \label{proofsec}

In this section we give the proof of the following
theorem assuming one key estimate, Theorem \ref{nov29.theorem1}, 
 that we prove in Section \ref{mainsec}.

\begin{theorem} \label{dimension}
If $\gamma$ is an $SLE_\kappa$ curve, $\kappa < 8$, then
with probability one for all $t_1 < t_2$,
\[  \hdim(\gamma[t_1,t_2]) = d . \]
\end{theorem}

We have already noted that
there is a $d_*$ such that with probability one for 
all $t_1 < t_2$, $\hdim(\gamma[t_1,t_2]) = d_*$.
The estimate $d_* \leq d$ is straightforward (see
Section \ref{uppersec}).
The hard work is showing that $d_* \geq d$ and this
is what we focus on. It suffices for us to consider
$t_1 = 1, t_2 = 2$.

We will use the notation from Section \ref{nov15sec}.
In order to show that $d_* \geq d$, we will show that the
conditions of Corollary \ref{Frostmancorollary2} are satisfied
with 
\[           \xi  = d(d-2) + 1\in(0,1)  . \]
As before, we set
\[              \beta = d - \frac 32 = \frac {\xi -1}{d} + \frac
  12. \]
For fixed positive integer
$n$ and integers $1 \leq j < k \leq n$, we let
\[   S = S_{j,n} = 1 + \frac{j-1}{n}, \;\;\;\;  T= T_{j,k,n}
 = \frac{k-j}{n},
\;\;\;\;  S+T = 1 + \frac{k-1}{n}, \]
\[   \eta(t) = \gamma(1+t) , \;\;\;\;  \eta(j,n)
   = \hat f_S(i/\sqrt n) = f_S( V_S + n^{-1/2} i). \]
Note that $1 \leq S \leq S+T \leq 2, 0 \leq T \leq 1$.
We let ${\cal F} =
{\cal  F}_S$ denote the $\sigma$-algebra
generated by $\{V_s: s \leq S\} = \{U_{T+s} - U_T:
s \leq S\}$  and ${\cal G} = {\cal G}_{S,T}$
the 
 $\sigma$-algebra
generated by $\{V_{S+t} - V_S: 0 \leq t \leq T\}
 = \{U_{t}: 0 \leq t \leq T\}$.  Note
that ${\cal F}$ and ${\cal G}$ are independent.
We let ${\cal F} \vee {\cal G}$ be the $\sigma$-algebra
generated by $\{U_{t}: 0 \leq t \leq S+T\}$.

Let us give an  idea of the strategy.  We
will define 
$ F(j,n) = n^{  1- \frac d2} \,
 |\hat f_S'(i/\sqrt n)|^d \, 1_{E} $
where $E = E_{j,n}$ is some event on which $ |
\hat f_S'(i/\sqrt n)| \approx n^{\beta}$.  The event $E$ will 
describe ``typical'' behavior when we weight the paths
by  $|\hat f_S'(i/\sqrt n)|^d $; in particular, it will
satisfy
\[   \E\left[ |\hat f_S'(i/\sqrt n)|^d \, 1_{E}\right] \asymp 
    \E\left[|\hat f_S'(i/\sqrt n)|^d\right].\]
To establish tightness and give
lower bounds on Hausdorff dimension,
we need to  consider correlations which means estimating
$\E[F(j,n) F(k,n)]$.  Suppose, for example, $j=k$.  If
we did not include the event $E$ then we would be estimating
\[   \E\left[|\hat f_S'(i/\sqrt n)|^{2d}\right] , \]
which is not of the same order
of magnitude
as $(\E|\hat f_S'(i/\sqrt n)|^{d}])^2$. Indeed, if we weight paths
by $|\hat f_S'(i/\sqrt n)|^{2d}	$, we do not concentrate
on paths with $ |
\hat f_S'(i/\sqrt n)| \approx n^{\beta}$ but rather on paths
with $ |
\hat f_S'(i/\sqrt n)| \approx n^{\beta'}$ for some $\beta' > \beta$.
However, when we include the $1_E$ term, we can write
roughly
\[  \E[ (|\hat f_S'(i/\sqrt n)|^d \, 1_{E})^2]
   \approx n^{\beta d} \,  \E[ |\hat f_S'(i/\sqrt n)|^d \, 1_{E}].\]
 
\begin{remark}
The notation might be
a  little confusing.  The time $S$ will always be of
order $1$.  The time $T$ is at most that but generally will be
much smaller.
\end{remark}

 \subsection{Defining the $F(j,n)$}

   We  define $F(j,n)$ to be the
 ${\cal F}$-measurable  random variable
\begin{equation}  \label{nov29.15}
    F(j,n) = n^{  1- \frac d2} \,
 \left|\hat f_S'(i/\sqrt n)\right|^d \, 1_{E_{j,n}} 
\end{equation}
for some ${\cal F}$-measurable event 
\[   E_{j,n} =  E_{j,n,1}  \cap \cdots \cap
   E_{j,n,6}, \]
which we will define now.  We define $F(k,n)$
similarly;  it will be $({\cal F} \vee {\cal G})$-measurable.

The event $E_{j,n}$ is defined in terms of the
solution of the time-reversed Loewner equation.
Let
 $h_t, \tilde h_t$
as in Lemma \ref{apr4.lemma2}.  We write
$Z_t = h_t(i/\sqrt n) - U_t = X_t + i Y_t,
\tilde Z_t = \tilde h_t(i/\sqrt n) - \tilde U_t =
  \tilde X_t + i \tilde Y_t. $
In particular,
\[ h_{s+T}(i/\sqrt n)  =  \tilde h_s (Z_T ) + U_T, \]
\[ h_{s+T}'(i/\sqrt n) = \tilde h_s'(Z_T )
   \; h_T'(i/\sqrt n). \]

\begin{remark}  The transformation $h_{T} $ is
${\cal G}$-measurable and the transformation $\tilde h_s$
is ${\cal F}$-measurable.  The random variable
$Z_T$ is ${\cal G}$-measurable. The
random variable $\tilde h_s'(Z_T)$ is neither
${\cal F}$-measurable nor ${\cal G}$-measurable.
  The key
to bounding correlations at times $S$ and $S+T$
is handling this random variable.
\end{remark}

We will study $h_t$ in detail in Section \ref{mainsec}.  
In this section, we will need a couple of very simple estimates
that follows immediately from the Loewner equation.  For
fixed $z$, $Y_t(z)$ is increasing in $t$ and
\begin{equation}  \label{dec8.6}
     \p_t Y_t(z) = \frac{a\, Y_t(z)}{|Z_t(z)|^2},\;\;\;\;
    \p_t Y_t(z)^2 = \frac{2a\,Y_t(z)^2}
   {X_t(z)^2 + Y_t(z)^2}. 
\end{equation}
In particular,
\begin{equation}  \label{dec8.7}
\Im(z)^2 + 2at 
 \,  \min_{0 \leq s \leq t} \frac{Y_t^2(z)}{|Z_t(z)|^2}  
  \leq Y_t(z)^2 \leq \Im(z)^2 + 2at . 
\end{equation}
Also, 
\begin{equation}  \label{dec8.5}
             |h_t(z) - z| \leq \frac{at}{\Im(z)} , \;\;\;\;
                 \left|\log 
  |h_t'(z)|\right| \leq \frac{at}{\Im(z)^2}.\end{equation}

The six events will depend on a subpower function
$\phi_0$ to be    determined later.   Given
$\phi_0$ we define the following events.

\begin{equation}  \label{nov16.1}
 E_{k,n,1} = \left\{   {Y_t  } 
 \geq   t^{\frac 12} \, \phi_0(1/t)^{-1}
 \mbox{ for all }  1/n \leq t \leq S + T
  \right\}. 
\end{equation}
\begin{equation}  \label{nov16.1.1}
   E_{k,n,2} = 
 \left\{  {Y_t } 
\geq    t^{\frac 12} \, \phi_0(nt)^{-1}
    \mbox{ for all }  1/n \leq t \leq S+ T
  \right\}. 
\end{equation}
\[
E_{k,n,3} = \left\{ |X_t  | \leq t^{\frac 12} \, \phi_0(1/t) 
 \mbox{ for all }  
1/n \leq t \leq S+ T\right\}, \]
\[ E_{k,n,4} =  \left\{ |X_t  | \leq t^{\frac 12} \, \phi_0(nt) 
 \mbox{ for all }   
1/n \leq t \leq S+ T\right\}. \]
 
\[ E_{k,n,5} = \left\{
  (nt)^{\beta}\, \phi_0(nt)^{-1}
  \leq   
 {\left|\,   h_t  '(  i/\sqrt n) \, \right|} 
 \leq  
 (nt)^{\beta }  \,\phi_0(nt) \mbox{ for all }  
1/n \leq t \leq S+ T\right\}
\]
\begin{equation}  \label{nov16.3}
 E_{k,n,6} = \left\{
  t^{-\beta} \, \phi_0(1/t)^{-1}
  \leq   
 \left|\frac{    h_{S+T}  '(  i/\sqrt n) }
{   h_t  '(  i/\sqrt n)  } \right|
 \leq  
 t^{-\beta }  \, \phi_0(1/t)
\mbox{ for all }  
1/n \leq t \leq S+ T\right\}.
\end{equation}
$E_{j,n,\cdot}$ are defined in the same way replacing $h_t,Z_t,S+T$
with $\tilde h_t, \tilde Z_t,S$.

\begin{remark}  What we would really like to do is define is an
event of the form
\[  Y_t \asymp t^{\frac 12}, \;\;\;\;
   |X_t| \leq c \, t^{\frac 12}, \;\;\;\;  |h_t'(i/\sqrt n)|
  \asymp (nt)^\beta, \]
for all $0 \leq t \leq S+T$.  
However, this is too strong a restriction if we want the event
to have positive probability (in the weighted measure).  What
we have done is modify this so that  quantities are comparable
for times near zero and for times near $S+T$ but the
error may  be larger for times in between (but still bounded by 
a subpower function).
\end{remark}

\begin{theorem} \label{nov29.theorem1}
There exist $c_1,c_2$ such that for all $t \geq 1/n$,
\begin{equation}  \label{dec2.10}
 \E\left[\left| h'_t(i/\sqrt n)\right|^d\right]
  = \E\left[ \left|h'_{tn}(i)\right|^d\right]
  \leq  c_2  \, (tn)^{\frac d2 - 1} . 
\end{equation}
Moreover there exists a power function $\phi_0$
 such that if
 $E_{j,n}$ is
  defined as above, then  
\begin{equation}  \label{dec2.11}
    \E\left[\left|\tilde h_S'(i/\sqrt n)\right|^d \,1_{E_{j,n}}\right] \geq c_1\, 
  \, n^{\frac d2 - 1}.
\end{equation}
\end{theorem}

\begin{proof}  See Section \ref{mainsec}. \end{proof}

\begin{remark} The  equality in \eqref{dec2.10}
follows immediately from scaling.
\end{remark}

\subsection{Handling the correlations}

  Theorem \ref{nov29.theorem1} discusses  the function
$\tilde h_t$ and the corresponding processes $\tilde X_t ,\tilde Y_t $
for a fixed value of $S$.  In this
section we assume  Theorem \ref{nov29.theorem1} and
show how to verify
the hypotheses of Corollary \ref{Frostmancorollary} for
$\xi$ as defined earlier and
some subpower function $\phi$.  Here $F(j,n)$ is
defined as in \eqref{nov29.15}.  The first hypothesis \eqref{nov15.h1.A}
follows immediately from \eqref{dec2.10}
so we will only need to
consider \eqref{nov15.h2.A}--\eqref{nov15.h4.A}.
Throughout this subsection $\phi$ will denote a subpower
function, but its value may change from line to line.

\subsubsection{The  estimate \eqref{nov15.h2.A}}

We first consider $j=k$.  Then
\[  \E[F(j,n)^2] = n^{2-d} \, \E\left[ \left|\tilde h_S '(i/\sqrt n)
  \right|^{2d}
    \, 1_{E_{j,n}}\right]. \]
But on the event $E_{j,n}$ we know that
$|\tilde h_S '(i/\sqrt n)| \leq n^\beta \, \phi(n)$.  Therefore,
using \eqref{dec2.10},
\[     \E[F(j,n)^2] \leq n^{2 -d +  \beta d} \,
    \E\left[ \left|\tilde h_S '(i/\sqrt n)\right|^d\right]
 \, \phi(n)  
        \leq n^\xi \, \phi(n).  \]

   We now assume $j < k$.  We
need to give an upper bound for
\[ \E[F(j,n) \, F(k,n)]
  = n^{2-d} \,
   \E\, \left[\;1_{E_{j,n}} \,   |\tilde h_S'(i/\sqrt n) |
  \, 1_{E_{k,n}} \,  |h_{S+T}'(i/\sqrt n)|\;\right]  . \]
Let 
$\tilde  E_{k,n} = \tilde E_{k,n,1}
\cap \tilde E_{k,n,3} $
where
$\tilde E_{k,n,j}$ is defined as $E_{k,n,j}$  
  except that $1/n \leq t \leq S+T$ is replaced with
$1/n \leq t \leq T$. Then $\tilde E_{k,n}$ is ${\cal G}$-measurable
and 
 $E_{k,n}  \subset \tilde E_{k,n}$. From
\eqref{nov16.1}, we can write
\[              h_{T+S}  '(i/\sqrt n)
   =  h_T  '(i/\sqrt n) \;   \tilde h_{S}  '(
     Z_T) . \]
  Therefore,
\[
n^{d-2} \, \E[F(j,n) \, F(k,n)] \leq \hspace{2.5in} \]
\begin{equation}  \label{nov29.9}
 \hspace{.5in}
 \E\left[\;1_{E_{j,n}} \;  |\tilde h_S '(i/\sqrt n) |
\;  |\tilde h_{S} '(
  Z_T) | \; 
1_{\tilde E_{k,n}} \;  | h_{T} '(i/\sqrt n) |\; \right]. 
\end{equation}
This is the expectation of
a product of five random variables.  The first two are
${\cal F}$-measurable and the last two are ${\cal G}$-measurable.
The middle random variable $|\tilde h_{S} '(
  Z_T)|$ uses information from both $\sigma$-algebras: the transformation
$\tilde h_S $ is ${\cal F}$-measurable but it is evaluated at
$Z_T$ which is ${\cal G}$-measurable.

We claim that it suffices to show
that   on the event $E_{j,n} \cap \tilde E_{k,n}$,
\begin{equation}  \label{nov29.8.1}
  \left|\tilde h_{S}  '(Z_T)\right|^d  \leq
   T^{-\beta d} \, \phi(1/T) , 
\end{equation}
for some subpower function $\phi$. Indeed, once we have established this we
can see that the expectation in \eqref{nov29.9} is bounded above by
\[  T^{-\beta d} \, \phi(1/T) \, 
 \E\left[  | \tilde h_S '(i/\sqrt n) |
\,  
  \left| h_{T} '(i/\sqrt n)\right|\; \right], \]
which by independence equals
\[  T^{-\beta d} \, \phi(1/T) \, 
 \E\left[   |\tilde h_S '(i/\sqrt n) |
\right] \, \E\left[\;
  | h_{T}  '(i/\sqrt n) |\; \right]. \]
Using \eqref{dec2.10}, we then have that this is bounded by
\[ T^{-\beta d} \, \phi(T) \,  n^{\frac d2 - 1} \,
                 (nT)^{\frac d2 - 1}  = T^{-\xi} \, \phi(1/T)
   = \left(\frac{n}{k-j}\right)^\xi \, \phi\left(\frac{n}{k-j}
  \right). \]
Hence, we only need to establish \eqref{nov29.8.1}.

On the event $\tilde E_{k,n}$,
we know that
\begin{equation}  \label{dec8.1}
   |X_T| \leq T^{1/2} \, \phi(1/T), \;\;\;\;\;
                     \phi(1/T)^{-1} \, T^{1/2}  
  \leq Y_T \leq  c \, T^{1/2} . 
\end{equation}
Using  Lemma \ref{nov29.lemma1} and \eqref{dec8.1}, we can see that
\[       | \tilde h_{S} '(Z_T)|^d  \leq
   \phi(1/T) \,  |\tilde h_{S} '(i \sqrt T)|^d . \]
We can write
\[                 \tilde h_S  '(i \sqrt T) =
   \hat h _{S-  T} '
(\tilde Z_T(i \sqrt T))\, 
             \tilde h_{  T}  '(i \sqrt T) , \]
where $\hat h$ is defined like $h,\tilde h$ except
using $S-T$ in place of $S+T,S$.
By \eqref{dec8.5}, we know that
\[              |\tilde  h_{ T} (i \sqrt T)
     |  \leq  c\, \sqrt T  , \;\;\;\;
       |\tilde h_{  T} '(i \sqrt T)| \leq c . \]
But on the event $E_{j,n}$,
\[
   |\tilde X_{  T}    | \leq T^{1/2} \,
  \phi(1/T), \;\;\;\;   T^{1/2} \, \phi(1/)^{-1} 
  \leq \tilde Y_{ T}  \leq c \,T^{1/2}, 
\]
from which we can see that $|\tilde Z_T(i \sqrt T))| 
\leq T^{1/2} \, \phi(1/T).$
Using Lemma \ref{nov29.lemma1} again we see that
\[     | \hat h_{S-T}  '( \tilde Z_T(i \sqrt T)) |
      \leq \phi(1/T) \,  |  \hat h_{S-T}  '( 
   Z_{2T}) |, \]
and the right-hand side is bounded by $\phi(1/T) \, T^{-\beta}$
by \eqref{nov16.3}.  This gives \eqref{nov29.8.1}.

\subsubsection{The estimate \eqref{nov15.h3.A}}

   Assume that $j < k$
 and that  $E_{j,n}$
and $E_{k,n}$ both occur. 
Then,
\[ \eta(k,n) = \hat f_{S+T}(i/\sqrt n) = h_{S+T} 
   (1/\sqrt n) +  U_{T+S}
  = \tilde h_S(Z_T) + U_S . \]
  Therefore,
\[   \eta(k,n) - \eta(j,n) =  \tilde h(Z_S)
              -\tilde h_S (i/\sqrt n). \]
By \eqref{nov16.1} and \eqref{nov16.1.1} we know that
\[             Y_T \geq 
 T^{\frac 12} \, \max\{\phi(nT)^{-1},\phi(1/T)^{-1}\}
   = T^{\frac 12}
    \max\left\{\phi(k-j)^{-1},
\phi\left(\frac n{k-j}\right)^{-1}\right\}. \]
Also, using \eqref{dec8.7}, we see that there is a $c$ such that 
\begin{equation}  \label{dec8.9}         Y_T - n^{-1/2}  > c \, Y_T. \
\end{equation}
 
Also,
\[     h_{S+T} '(i/\sqrt n) =  h_{T} '(i/\sqrt n)
              \, \tilde h_S'(Z). \]
By (\ref{nov16.3}) we know that
\[    | \tilde h_S '(Z_{T} ) | \geq 
  T^{-\beta } \, \phi\left(\frac n{k-j}\right)^{-1}  . \]
The Koebe $(1/4)$-Theorem tells us that the image of the ball
of radius $ c Y_T$ about $Z_T$
under $\tilde h_S$ contains a ball ${\cal B}$ of raidus
\[              \frac {cY_T}4  
  \,     |\tilde h'(Z_T)| 
   \geq c \, T^{\frac 12 + \beta  }  
    \, \phi\left(\frac n{k-j}\right)^{-1}    \]
about $\eta(k,n)$.  But \eqref{dec8.9} tells us
that $i/\sqrt n$ is not in the ball of
radius  $c YT$ about $ZT$ and
hence $\eta(j,n) \not \in {\cal B}$.  Therefore,
\[   \left| \eta\left(k,n\right) - \eta \left(
   j,n \right)\right|^d \geq  T^{\frac d2 + d\beta }\,
\phi\left(\frac n{k-j}\right)^{-1}
  =   T^{\xi - 1 }    
 \, \phi\left(\frac n{k-j}\right)^{-1}  . \]

\begin{remark} Note that we do not expect that last
estimate to hold for all $k,j$, especially for $\kappa > 4$
for which $SLE_\kappa$ has double points.  The restriction
to the event $E_{j,n} \cap E_{k,n}$ is a major
restriction.
\end{remark}

\subsubsection{The estimate \eqref{nov15.h4.A}}

This estimate follows from \eqref{jan8.3}.   This lemma
 was essentially proved by Rohde and Schramm when they
proved existence of the curve.  We rederive this
result in Section \ref{curvesec}.

 \begin{remark}
In fact,
 we do not need this  to prove our result.  On the
event $E_{j,n}$, we have $|\hat f'_S(i/\sqrt n)| \approx n^{\beta}$.
Therefore, using the Koebe $(1/4)$-Theorem we can conclude on
this event that for every $\epsilon > 0$, there is a $c$
such that 
\[           \dist\left[\hat f_S(i/\sqrt n) , \gamma[0,s]
   \cap \R\right] \leq c\, n^{\beta + \epsilon} \, n^{-1/2}  =
  c\, n^{d + \epsilon - 2} . \]
Since $d<2$, this goes to $0$ for $\epsilon$ sufficiently small.
\end{remark}

\section{Derivative estimates for reverse flow}  \label{mainsec}

Theorem \ref{nov29.theorem1} is a statement about the flow 
$h_t(z)$ for a fixed $z \in \Half$.  By appropriate change
of variables,  this flow can be understood
in terms of  a one-dimensional diffusion.
Throughout this section we fix $a$.  If $a > 1/4$,
we
let $d,\beta,\xi$ be as above. 
All constants in this section
may depend on $a$. 
We will consider solutions of the
time-reversed Loewner equation 
\begin{equation}  \label{nov16.12}
   \dot{h}_t(z) =  \frac{a}{ U_t - h_t(z)  },
\;\;\;\;  h_0(z) = z, 
\end{equation}
where $U_t = - B_t$ is a standard Brownian motion. The
scaling properties of Brownian motion imply that for each
$r > 0$, the distribution of the random function
$z \mapsto r^{-1} \, h_{r^2t}
 (rz)$ is the same as the distribution of $z \mapsto
h_t(z)$.  In particular, the distribution of $h_{r^2t}'(rz)$
is the same as that of $h_t'(z)$.  We
define $X_t(z),Y_t(z)$ by
\[           h_t(z) - U_t = X_t(z) + i Y_t(z). \]
   We
now restate Theorem \ref{nov29.theorem1} in a scaled
form.

\begin{theorem}  \label{dec1.theorem2}
Suppose $a > 1/4$. 
There exist
$0 < c_1,c_2 < \infty$  and a subpower function $\phi$    such
that the following holds.   
Let $ X_t=X_t(i), Y_t=Y_t(i)$,
and let $E = E(\phi,t)$ be the event that for all $1 \leq s
\leq t$, 
\[               {\sqrt s}\, \max\left\{\frac{1}{\phi(s)} ,
  \frac{1}{\phi(t/s)}\right\} \leq 
    Y_s  \leq \sqrt{2as +1}
      , \]
\[     \frac{s^\beta}{\phi(s)} \leq |h_s'(i)|
    \leq  s^\beta \, \phi(s)
   , \;\;\;\;\;
       \frac{(t/s)^\beta}{\phi(t/s)} \leq 
  \frac{|h_t'(i)|}{|h_s'(i)|}
    \leq  (t/s)^\beta \, \phi(t/s)
   , \]
\[  |X_s| \leq   \sqrt s \, \min\left\{\phi(s),\phi(t/s)\right\}\
  . \]
Then,  
for all $t \geq 1$, 
\begin{equation}  \label{dec4.10}   c_1 \, 
 t^{\frac{d}{2} - 1}
 \leq  \E\left[\left|h_t'(i)\right|^d \, 1_E \right]   \leq
  \E\left[\left|h_t'(i)\right|^d\right] \leq c_2 \, t^{\frac  d2 -1}.  
\end{equation}
\end{theorem}

We will prove a more general result than this theorem.
Theorem \ref{dec1.theorem2} is a special
case of
Propositions \ref{apr3.prop} and \ref{general2} which treat
the upper and lower bounds, respectively.  Careful
examination of the  proofs will  
show that one can choose
\[  \phi(x) = C\, \exp\left\{[\log (x+1)]^{1/2}
     [\log \log(x+2)]^u \right\} \]
for some $C,u < \infty$. However, we do not need the exact form
of $\phi$ in this paper and we will not keep track of this.

Before going into details, let us give a quick overview
of the strategy.  The basic idea is to consider
the following martingale
\[             M_t = |h_t'(i)|^d \, Y_t^{2-d}
   \, \sqrt{R_t^2 + 1}, \;\;\;\;\;  R_t = X_t/Y_t . \]
The first two
terms in $M_t$ are (random)
differentiable functions of $t$; only the
third has nontrivial quadratic variation.  We expect
that
typically
$Y_t \approx t^{1/2}$ and $ R_t^2 + 1  
\approx 1$.  This suggests but does not prove   the estimate
 $\E[ |h_t'(i)|^d]
\asymp t^{\frac d2 - 1}$.  To estimate something
like  $\E[|h_t'(i)|^d \, 1_E ] $, we consider a
similar expectation $\E[M_t \, 1_E]$.  The Girsanov
 theorem tells us that we can compute
expectations like these by computing the probability
of $E$ under a different measure which corresponds
to changing the SDE by adding to the drift.  In our case,
the derived equation is fairly simple to analyze and
this allows us to prove the theorem.  
The computation is made somewhat
simpler by changing time so that
$Y_t$ grows deterministically.  This time change
 was used
by Rohde and Schramm (\cite{RS}) and analogues of it
have appeared in many places. 
This time
parametrization works very well when considering $h_t(z)$
for a single $z$, which is what we are
doing in this section. However, the time change depends on $z$,
so it is not so easy to use it for studying
the joint distribution of $(h_t(z),h_t(w))$ for
distinct $z,w$.
 
Sections \ref{secone} and \ref{sectwo} prove results that
will be used in \cite{LShef}.  These sections can be skipped
if the reader is only interested in Theorem
\ref{dec1.theorem2}. Also, Theorem \ref{dec1.theorem2}
needs Proposition \ref{apr3.prop} 
only for $x=0$ for which the proof
simplifies.

\subsection{An important martingale}  \label{firstsec}

Here we do some straightforward computations and introduce
the important martingales for the reverse Loewner flow.
Let $z = y(x+i) \in \Half$.
Let  $Z_t = Z_t(z)   = h_t(z) - U_t = X_t + iY_t$.
  The time-reversed Loewner equation
\eqref{nov16.12}
can be written as
\begin{equation}  \label{oct30.2}
 dX_t = - \frac{aX_t}{X_t^2 + Y_t^2} \, dt + dB_t, 
\;\;\;\; \p_t Y_t = \frac{aY_t}{X_t^2+Y_t^2} 
   = Y_t \, \frac{a(X_t^2 + Y_t^2)}{(X_t^2 + Y_t^2)^2}. 
\end{equation}
Note that $Y_t$ is strictly increasing and 
$\p_t(Y_t^2) \leq 2a$ which implies
\begin{equation}  \label{ybound}
           Y_t^2 \leq  {2at +y^2}  
 .  
\end{equation}
Differentiating \eqref{nov16.12} with respect to $z$ gives
\[    \p_t[\log h_t'(z)] = \frac{a}{Z_t^2}  , \;\;\;\;\;
 h_t'(z) = \exp\left\{\int_0^t \frac{a}{Z_s^{2}} \, ds\right\}. \]
\[     |h_t'(z)| = \exp\left\{\int_0^t \Re [aZ_s^{-2}] \, ds\right\}
    = \exp\left\{
\int_0^t \frac{a(X^2_s-Y_s^2)}{(X_s^2 + Y_s^2)^2} \, ds\right\}, \]
\begin{equation}  \label{apple.1}
  \p_t  |h_t'(z)| =  |h_t'(z)| \,  \frac{a(X^2_t-Y_t^2)}{(X_t^2 + Y_t^2)^2}
.  \end{equation}
Let \[ \distb_t = \frac{|h_t'(z)|}{Y_t},\;\;\;\; R_t = \frac{X_t}{Y_t}.\]
It\^o's formula and the chain rule give the following.
\begin{equation}  \label{oct30.2.extra}
 d(R_t^2 + 1)^{r/2}= (R_t^2+1)^{r/2}
  \, \left[\frac{ (-2ar + \frac{r^2}{2} - \frac r2 )X_t^2 +   \frac r2 
  \, Y_t^2}{(X_t^2 + Y_t^2)^2} \, dt + \frac{rX_t}{X_t^2 + Y_t^2 }
\, dB_t \right],
\end{equation}
\[ 
  \p_t[\distb^k_t]
 =  \distb_t^k 
 \,  \frac{-2ak\,Y_t^2}
      {(X_t^2 + Y_t^2)^2} . 
\]
In particular $\distb_t$ decreases in $t$ which combined
with \eqref{ybound} gives
\begin{equation}  \label{jan10.1}
 |h_t'(z)| \leq y^{-1} \, Y_t(z) \leq \sqrt{2a(t/y^2) + 1} . 
\end{equation}

\begin{proposition}    \label{prop1.apr29}
If $r,\theta,\lambda \in \R, z \in \Half$ and
\[    N_t = |h_t'(z)|^\lambda \, Y_t^{\frac \theta a-\lambda}
  \, (R^2_t+1)^{\frac r2} = \distb_t^{\lambda}
            \, Y_t^{ \frac \theta a} \, (R_t^2 + 1)^{\frac r2}. \]
Then,
\begin{equation}  \label{apr29.1}
 dN_t = N_t \, \left[\frac{j_X \, X_t^2 + j_Y \, Y_t^2}
   {(X_t^2 + Y_t^2)^2} \, dt + \frac{r \, X_t}{X_t^2 + Y_t^2}
  \, dB_t\right], 
\end{equation}
where
\[             j_X = \frac{r^2}{2} - \left(2a +
  \frac{1}{2}\right) \, r + \theta , \]
\[    j_Y = \frac r2 - 2a\lambda + \theta . \]
If $r \in \R$, and 
\[  \lambda = \lambda(r) =  r \,\left(1 + \frac 1{2a}\right)
   - \frac{r^2}{4a} , \]
then
\begin{equation}  \label{nov3.10}
   M_t=  \distb_t^\lambda  \, Y_t^{2\lambda - \frac{r}{2a}} \,
  ({R_t^2 + 1})^{r/2} =  {|h_t'(z)|^\lambda \, Y_t^{r-\frac{r^2}{4a}}}\,
 ({R_t^2 + 1})^{r/2}
,
\end{equation}
is a martingale satisfying
 \begin{equation}   \label{nov1.2}
           dM_t =  \frac{r\,X_t}{X_t^2 + Y_t^2} 
  \, M_t \, dB_t. 
\end{equation}
\end{proposition}

\begin{proof}
The computations \eqref{apr29.1} -- \eqref{nov1.2} 
are straightforward using the product rule and
It\^o's formula.  Therefore,  
$M_t$ is a nonnegative
local martingale satisfying \eqref{nov1.2}.
We will now show how we can use the
Girsanov theorem to conclude that
 $M_t$ is a martingale.  
 Suppose $M_t$ is a nonnegative,  continuous
local martingale
satisfying
\[                dM_t = J_t \, M_t \, d B_t, \]
and, for ease, assume that $M_0 = 1$.  
For any $n \geq 1$, let
$\tau_n = \inf\{t: M_t \geq n\}.$  Then $M^{(n)}_t := M_{t \wedge
\tau_n}$ is a uniformly bounded martingale satisfying
\[       dM_t^{(n)} = J_t^{(n)} \, M_t^{(n)} \, d B_t, \]
where $J_t^{(n)} = J_t \, 1\{\tau_n > t\}$.  Let $Q_t^{(n)}$
be the probability measure given by weighting by $M_{t}^{(n)}$,
i.e., for every $\F_t$-measurable event $A$,
\[    Q_t^{(n)}[A] = \E\left[M_t^{(n)} \, 1_A\right]. \]
The Girsanov theorem tells us that
\[    W_{t}^{(n)} = B_t - \int_0^t \, J_{s}^{(n)} \, ds, \]
is a standard Brownian motion with respect to the measure
 $Q_t^{(n)}$, or equivalently,
\begin{equation}  \label{dec7.8}
     dB_t = J_{t}^{(n)} \, dt + dW^{(n)}_t.
\end{equation}
The same holds for $M_t = M_t^{(\infty)}$ provided that $M_t$
is a martingale.  In order to see whether or not $M_t$ is a
martingale, one needs to check whether or not some mass
``disappears'' as $n \rightarrow \infty$.  If for every
fixed $t$, 
\[      \lim_{n \rightarrow \infty} Q_t^{(n)}\{\tau_n \leq t\}
   = \lim_{n \rightarrow \infty} \E[M_{\tau_n} \,;\,\tau_n \leq t]
  = 0 , \]
then no mass disappears.   Since the stopped process
satisfies the equation \eqref{dec7.8}, we can see that $M_t$
is a martingale provided that the system
\begin{equation}  \label{dec7.9}
    dB_t =J_t \, dt + dW_t, \;\;\;\;\;
   dM_t =  J_t \, M_t\, dB_t , 
\end{equation}
has no explosion in finite time, i.e., for each $\epsilon >0,
t < \infty$, there is an $N< \infty$ such that if $B_t,M_t$
satisfies \eqref{dec7.9}, then with probability at least
$1-\epsilon$, $\sup_{s \leq t} (|B_s|+ M_s) \leq N$.  
In order to simplify the notation, when we weight by the local
martingale, then we will say that Girsanov theorem implies that
$B_t$ satisfies
\[  dB_t =J_t \, dt + dW_t. \]
To be precise, this
 should be interpreted in terms of stopping times as 
in \eqref{dec7.8}.

  We now consider the case at hand.
   If we use Girsanov's theorem and weight $B_t$
by the martingale $M_t$ as in \eqref{nov3.10},
 then  
\[   dB_t = \frac{r\,X_t}{X_t^2 + Y_t^2} \, dt + d\tilde B_t \]
where $\tilde B_t$ is a Brownian motion with respect to the new measure.
  In other words, with respect to the new measure
\begin{equation}  \label{nov1.3}
 dX_t  = \frac{(r-a)\, X_t}{X_t^2 + Y_t^2} \, dt + d\tilde B_t. 
\end{equation}
Note that $Y_t$ and $|h_t'(z)|^d$ are differentiable
quantities so their equations do not change in the
new measure.
By comparing to a Bessel equation,
it is easy to check that if $X_t$ satisfies
the above equation then there is no explosion
in finite time.   Since
\[  M_t =   \distb_t^\lambda  \, Y_t^{2\lambda  - \frac{r}{2a}} \,
  ({R_t^2 + 1})^{r/2} =  \distb_t^\lambda  \, Y_t^{2\lambda  - \frac{r}{2a}
  - r} \,
  ({X_t^2 + Y_t^2})^{r/2}
  , \]
and $\distb_t \leq 1/y, Y_t \leq \sqrt{2at + y^2}$, 
  this also
shows that $M_t$ has no explosion in finite time.   This is
enough to conclude that $M_t$ is a martingale.
\end{proof}

The particular values of the parameter that are most
important for Theorem \ref{dec1.theorem2} are
\[   r=1,\;\;\;\;\; \lambda =d,\;\;\;\;\; r - \frac{r^2}{4a}
   = 2-d. \]

\subsection{Simple consequences}  \label{secone}

Proposition \ref{prop1.apr29} suggests but does not prove 
the upper bound in 
Theorem \ref{dec1.theorem2}.  We have
\[         \E\left[|h_t'(i)|^\lambda \,  Y_t^{r- \frac{r^2}{4a} }\,
(R_t^2 + 1)^{r/2}
  \right] = 1.\]
One might hope
that the
 typical value of  $R_t^2 +1$ is of order $1$ and the
typical value of $Y_t$ is of order $\sqrt t$.  If this were true
then
the expectation would be comparable to 
\[        \E\left[|h_t'(i)|^\lambda \, 
 (\sqrt t)^{r- \frac{r^2}{4a}}\right], \]
and we would have the upper bound.  
In Proposition \ref{apr3.prop} we show this is true for a range
of positive $r$.  The range includes $r=1$ if $a > 1/4$.

In this section we consider the simpler problem of giving
upper bounds for 
\[         \E\left[|h_t'(z)|^\lambda; Y_t  \geq \epsilon \,
   \sqrt t\right]. \]
The results in this section will be used in \cite{LShef}.
By scaling, we see that
\[     \E\left[|h_t'(y(x+i))|^\lambda; Y_t(y(x+i)) \geq \epsilon \,
   \sqrt t\right] = \E\left[|h_{t/y^2}'(x+i)|^\lambda; 
Y_t(x+i) \geq \epsilon \,
   \sqrt {t/y^2}\right], \]
so it suffices to consider the case $z = x+i$. We restrict our
consideration to $0 \leq r \leq 2a + 1$.  In this interval, $
r \mapsto \lambda(r)$
is one-to-one and we can write the martingale \eqref{nov3.10}  as
\[    M_t = |h_t'(z)|^{\lambda} \, Y_t^{\zeta(\lambda)}
  \, \left(R_t^2 + 1 \right)^{r(\lambda)}, \]
where 
\begin{equation}  \label{apr3.1}
  r=  r(\lambda) = 2a +1 - 2a \sqrt{\left(1 + \frac 1{2a}
\right)^2 - \frac \lambda
  a }, 
\end{equation}
\begin{equation} \label{apr3.2}
 \zeta=
 \zeta(\lambda)  = r - \frac{r^2}{4a} = \lambda - \frac{r}{2a}
         = \lambda +  \sqrt{\left(1 + \frac 1{2a} \right)^2 - \frac \lambda
  a } - 1 - \frac 1{2a} . 
\end{equation}
Note that $\lambda \mapsto \zeta(\lambda)$ is strictly concave.
We define $\lambda_c$ by $\zeta'(\lambda_c) = -1$;
in other words, 
\[  r(\lambda_c) = 2a + \frac 12, \;\;\;\;
    \lambda_c = a + \frac{3}{16a} +1 , \;\;\;\;
  \zeta(\lambda_c) = a - \frac 1{16a} .\]

\begin{proposition}  Suppose $0 \leq \lambda \leq
\lambda_c =  a + \frac{3}{16a} +1$.
Then for every $x \in \R, \epsilon > 0$, if $Y_t = Y_t(x+i)$,
\[    \E\left[|h_t'(x+i)|^\lambda; Y_t \geq
  \epsilon \sqrt t\right] \leq (x^2+1)^{r/2} \, \epsilon^{-\zeta }
               \, t^{-\zeta/2} , \]
where $r,\zeta$ are as
defined in \eqref{apr3.1} and \eqref{apr3.2}.
\end{proposition}

\begin{proof}
\begin{eqnarray*}
 \E\left[|h_t'(x+i)|^\lambda; Y_t  \geq
  \epsilon \sqrt t\right] & \leq & \epsilon^{-\zeta }
               \, t^{-\zeta/2} \,
    \E\left[|h_t'(x+i)|^\lambda \, Y_t ^{\zeta}\right]\\
      & \leq & \epsilon^{-\zeta }
               \, t^{-\zeta/2} \, \E[M_t]\\
      & = & \epsilon^{-\zeta }
               \, t^{-\zeta/2} \, (x^2 + 1)^{r/2}. 
\end{eqnarray*}
\end{proof}

The preceding proof did not use the fact that $\lambda \leq \lambda_c$.
However, for $\lambda > \lambda_c$ we can get a better estimate
as given in the next proposition.

\begin{proposition} Suppose
$\delta > 0$.  Then,
for every $x \in \R, \epsilon > 0$, if $Y_t = Y_t(x+i)$,
\[    \E\left[|h_t'(x+i)|^{a + \frac{3}{16a} +1 +
  \delta}; Y_t  \geq
  \epsilon \sqrt t\right]  \leq
        (2at + 1)^{\frac{\delta}{2}} \,
              \epsilon^{\frac 1{16a} -  a }
  \, t^{\frac 1{32a} - \frac a2} \, (x^2 + 1)^{a +
 \frac 14}. \]
\end{proposition}

\begin{proof}  We know from 
\eqref{jan10.1} that $|h_t'(x+i)| \leq
  \sqrt{2at + 1}$.  Therefore,
\[ \E\left[|h_t'(x+i)|^{\lambda_c + \delta}; Y_t  \geq
  \epsilon \sqrt t\right] \leq
    (2at + 1)^{\delta/2}
           \,  \E\left[|h_t'(x+i)|^{\lambda_c}; Y_t  \geq
  \epsilon \sqrt t\right]. \]
\end{proof}

\begin{remark}
Roughly speaking, we have shown
\[  \E\left[|h_t'( i)|^{\lambda }; Y_t  \geq
  \epsilon \sqrt t\right] \leq c(\epsilon) \, t^{\beta(\lambda)},\]
where
\[   2\beta(\lambda) = \left\{\begin{array}{ll} -\zeta(\lambda),
   & \lambda \leq \lambda_c ,\\
  -\zeta(\lambda_c)  + (\lambda - \lambda_c),
  & \lambda  \geq \lambda_c. \end{array} \right. \]
Since $\zeta$ is strictly concave with $\zeta'(\lambda_c) = -1,$
\[           -\zeta(\lambda_c)  + (\lambda - \lambda_c)
    < - \zeta(\lambda) , \;\;\;\;\;  \lambda > \lambda_c. \] 
This is a standard ``multifractal'' analysis of a moment.
As $\lambda$ increases to $\lambda_c$, the expectation
of $|h_t'(i)|^\lambda$ concentrates on the event that
$|h_t'(i)| \approx t^{- \zeta'(\lambda)/2}.$  When $\lambda$
reaches $\lambda_c$, the expectation concentrates on the
event $|h_t'(i)| \approx t^{1/2}$.  However, we know that
$|h_t'(i)| \leq ct^{1/2}$ so all higher powers $\lambda$
also concentrate on this event.  This makes $\zeta$ a linear
function for $\lambda \geq \lambda_c$.  
\end{remark}

\begin{example}  In \cite{LShef}, we will need to consider the
case $a>1/4, \lambda = 2d = 2 + \frac 1{2a}$. 
We will consider two subcases.

\begin{itemize}

\item $  5/4 \leq a < \infty$. 
 In this range $2d \leq \lambda_c$. We have
\[  r = r(2d) = 2a + 1  - 2a \sqrt{1 - \frac 1a - \frac{1}{4a^2} } , \]
\[  \zeta = \zeta(2d) =   1 + \sqrt{1 - \frac 1a - \frac{1}{4a^2} }. \]
\begin{equation}  \label{scottpaper1}
 \E\left[|h_t'(x+i)|^{2d}; Y_t  \geq
  \epsilon \sqrt t\right] \leq c_\epsilon   \,
 (x^2 + 1)^{r/2} \, 
  t^{-\zeta/2}. 
\end{equation}

\item $ 1/4 < a \leq 5/4$.  In this range
\[   2d =   \lambda_c + \left[1 + \frac 5{16a} -a\right], \]
where the term in brackets is nonnegative, and hence
\begin{equation}  \label{scottpaper2}
 \E\left[|h_t'(x+i)|^{2d}; Y_t  \geq
  \epsilon \sqrt t\right] \leq c_\epsilon\,   (x^2 + 1)^{a +
 \frac 14}  \, (t+1)^{-\tilde \zeta/2} 
\end{equation}
 where
\[  \tilde \zeta = \zeta(\lambda_c) - \left[ 1 
  + \frac{5}{16a} - a\right]
  = 2a - \frac{3}{8a} - 1. \]
Note that $\tilde \zeta = 0$ if $a = 3/4$ and $\tilde \zeta < 0$ for
$1/4 < a < 3/4$. 

\end{itemize}

\end{example}

\subsection{Change of time}  \label{changesec}

It is convenient to 
 change time so that $\log Y_t$
grows linearly. This  converts
the two-variable process $(X_s,Y_s)$ into a  one-variable
process.  Let
\[  \sigma(t) = \inf\{s: Y_{s} = e^{at}\}, \;\;\;\;
 \hat Y_t = Y_{\sigma(t)} = e^{at} ,\;\;\;\;
 \hat X_t  = X_{\sigma(t)}, \;\;\;\;
   K_t = R_{\sigma(t)} = e^{-at}\, \hat
  X_t .  \] 

\begin{lemma}\begin{equation}  \label{nov6.1}
  \dot \sigma(t) =  \hat X_t^2 + \hat Y_t^2  =
  \hat X_t^2 + e^{2at}  , \;\;\;\;
  \sigma(t) =  \int_0^t  (\hat X_s^2 + e^{2as})\, ds
  = \int_0^t  e^{2as} \, ( K_s^2 + 1)\, ds. 
\end{equation}
\end{lemma}

\begin{proof}
Since $\p_t \hat Y_t =  a\, \hat Y_t$, we have by (\ref{oct30.2}),
\[   a\, \hat Y_t = \p_t \hat Y_t =
      \dot{Y}_{\sigma(t)}  \, \dot \sigma(t) =
               \frac{  a\,\hat Y_t}{\hat X_t^2 + \hat Y_t^2}
   \, \dot \sigma(t) . \] \end{proof}

Using \eqref{oct30.2} we get  
\[  d \hat X_t = -a\,\hat X_t \, dt + \sqrt{\hat X_t^2 + e^{2at}} \,
 dB_t , \]
\begin{equation}  \label{oct30.4}
   dK_t = -2a \, K_t \, dt + \sqrt{K_t^2 + 1} \, dB_t. 
\end{equation}  From
\eqref{apple.1}  we see that
\[     \p_t |h_{\sigma(t)}'(z)|  
  = |h_{\sigma(t)}'(z)| \, \frac{a\, (\hat X_t^2 - 
  \hat Y_t^2)}{\hat X_t^2 + \hat Y_t^2} =  
   a\,|h_{\sigma(t)}'(z)| \,\left[1 - \frac{2}{K_t^2 + 1}
  \right]   ,  \] 
and hence,
\[      e^{-at} \, |h_{\sigma(t)}'(z)| = \exp\left\{-        
  a\int_0^t  \frac{2 \, ds}{K_s^ 2 + 1} \right\}.\]

\subsection{The SDE (\ref{oct30.4})}

We will study the
equation \eqref{oct30.4} from last subsection
which we write as
\begin{equation} \label{nov20.10}
 dK_t = \left(\frac 12  -  q -r 
  \right) \, K_t \, dt + \sqrt{K_t^2 + 1} \, dB_t, 
\end{equation}
where $q = 2a + \frac 12 - r $.  In this section, we view
$q,r$ as the given parameters and we define $a$ by
 $2a = q + r - \frac 12$.
We will be primarily interested in $q > 0$. 
Note that 
\[    q > 0 \;\;\; \Longleftrightarrow \;\;\; r<
  2a + \frac 12 . \]
We write $\E^x,\Prob^x$ for expectations and probabilities
assuming $K_0 = x$.  If the $x$ is omitted, then it is assumed
that $K_0 = 0$.

Let
\[  L_t = \int_0^t  \frac{K_s^2-1}
   {K_s^2 + 1}  \, ds = t - \int_0^t \frac{2\,ds}{K_s^2 + 1}. \]
Note that $-t \leq L_t \leq t$ and if $p \in \R$, 
\begin{equation} \label{nov19.4}
  \p_t[e^{p(L_t-t)}] = e^{p(L_t - t)} \, \frac{-2p}{K_t^2 + 1} , 
\end{equation}
As in \eqref{nov6.1}, we let
 \begin{equation}  \label{nov26.10}
 \sigma(t)  
  = \int_0^t e^{2as} \, [K_s^2 + 1]\, ds
\geq \int_0^t e^{2as}   \, ds
      = \frac{ 1}{2a} \, \left[e^{2at} -1
  \right],
\end{equation}
Since  $\sigma$ is strictly increasing, we can define $\rho =
\sigma^{-1}$, and the last inequality implies  
 that there is a $C'$ such that
\[  \rho(e^{2at}) \leq  t + C'. \]
We let
\[ \theta = \theta_q(r) =  
\frac r2 + qr + \frac{r^2}2 
  , \]
\[   N_t = e^{\theta (L_t-t)/2} \, 
e^{( \theta - \frac r2)t} \, ({K_t^2+1})^{r/2}
  . \]
  Although all the quantities above are
defined in terms of $K_t$, it is useful to note
that if $\theta = 2 \lambda a$, then
in the notation of   the previous section,
\[           |h_{\sigma(t)}'(z)| = e^{a L_t},
\;\;\;\; Y_{\sigma(t)} =e^{at} ,\; \;\;\;
   \frac{X_{\sigma(t)}}{Y_{\sigma(t)} }= K_t , \]
\[  M_{\sigma(t)} = e^{a \lambda L_t  }
    \, e^{at(r - \frac{r^2}{4a})}
  \, [K_t^2+1]^{r/2} =N_t .\]
 Since $N_t$ is $M_t$ sampled 
at  an increasing family of
 stopping times,  the next proposition
is no surprise.

\begin{proposition}  \label{prop.dec1}
 $N_t$
is a positive martingale satisfying
\begin{equation}  \label{nov19.2}
   dN_{t } =   N_{ t} \, \frac{rK_t}{\sqrt{K_t^2 + 1}}
   \, dB_t, \;\;\;\;\; N_0 = (x^2 + 1)^{r/2}.
\end{equation}
In particular,
\[ 
\E^x[e^{\theta L_t/2}\, ({K_t^2 + 1})^{r/2}  ] =   
   ({x^2 +1})^{r/2} \, e^{\zeta t/2}, \]
where 
\[ \zeta = \zeta(r) = r-\theta = \frac r 2 -qr - \frac{r^2}{2}. \]
\end{proposition}

\begin{proof} We have already noted that $N_t \geq e^{-t/2}$.
It\^o's formula gives
\[ d\sqrt{K_t^2 +1} = \sqrt{K^2_t+1} \, \left[
  \frac{K_t^2(\frac 12- r- q ) + \frac 12}{K^2_t+1}
  \,dt + \frac{K_t}{\sqrt{K_t^2+1}} \, dB_t \right],\]
\[  d (K_t^2+1)^{r/2} = (K_t^2+1)^{r/2} \, \left[
           \frac{(-qr-\frac{r^2}{2}) K_t^2 + \frac r2}{
     {K_t^2 + 1} }\, dt +\frac{rK_t}{\sqrt{K_t^2+1}} \, dB_t \right].\]
Using this and  (\ref{nov19.4}), we see that
 $N_{t}$ 
satisfies (\ref{nov19.2}).
If we  use Girsanov's theorem, the weighted paths satisfy
\begin{equation}  \label{oct30.4.alt} 
  dK_t = \left( \frac 12 - q \right) \, K_t 
\, dt + \sqrt { K_t^2 + 1} \, dW_t,  
\end{equation}
where $W_t$ is a standard Brownian motion in the
new measure. It is straightforward to
show (see Lemma \ref{apr2.lemma0}) that
this  equation does not have explosion in finite time,
and hence  we  can see that $N_t$
is actually a martingale from which we conclude the
final assertion.
\end{proof}

 We let 
$\tilde \E^x, \tilde \Prob^x$ denote probabilities with respect to the new measure, i.e.,
if $A$ is an event measurable with respect to
 $\{B_s: s \leq t\}$, then
\[            \tilde \Prob^x(A) = \tilde \E^x[1_A] =
\E^x[N_t \, 1_A]. \]
If $x$ is omitted, then $x=0$ is assumed.
Then $K_t$ satisfies \eqref{oct30.4.alt}
 where $ W_t$ is a standard Brownian motion
with respect to $\tilde \Prob $.

\subsection{The  SDE \eqref{oct30.4.alt}}  \label{SDEsec}

In this subsection we study the  one-variable SDE
\eqref{oct30.4.alt}.  
Note that if
we do a time change to \eqref{oct30.4.alt}, we get
the equation
\begin{equation}  \label{may1.1}
      d \hat K_t = \left( \frac 12 - q \right) \,
   \frac{\hat K_t}{\hat K_t^2 + 1} \, dt + d \hat W_t, 
\end{equation}
which is very similar to the standard Bessel equation.  In
fact, \eqref{may1.1} looks like a Bessel equation  
when $K_t$ is large but is better behaved than the Bessel
equation for $K_t$ near zero.  In analogy,
we expect the behavior of \eqref{oct30.4.alt} to have
three distinct
 regimes: $q > 0, q = 0, q<0$.  We will consider only the $q > 0$
case  for which the
process is positive recurrent.
We define
\begin{equation}  \label{apr2.1}
          \mu = \mu_q = \frac{1-2q}{1+2q}. 
\end{equation}

\begin{lemma} \label{apr2.lemma0}
For every $q > 0$,
 there is a $\delta > 0$ such that if 
$K_t$ satisfies  \eqref{oct30.4.alt} and
$x^2 \geq 1$,
\begin{equation}  \label{nov17.3}
  \Prob^x\left\{  \frac{x^2}{2} 
  \leq  K_s^2 \leq  2x^2, \; 0 \leq s \leq \delta  
 \right \} \geq \frac 12 .
\end{equation} 
In particular, the equation
does not have explosion in finite time.
\end{lemma}

\begin{proof}  Straightforward and left to the reader.
\end{proof}

\begin{lemma}  \label{apr2.lemma1}
Suppose $K_t$
satisfies \eqref{oct30.4.alt} with $q > 0$.  Let
\[  \phi(x)  =  
  \int_0^{|x|}  (s^2 + 1) ^{q-\frac 12 } \, ds,\] 
\[ \rho = \rho(y) = \inf\{t: K_t^2 = 0 \mbox{ or } K_t^2 \geq y^2\}. \]
Then, $\phi(K_{t \wedge \rho})$ is a martingale.  In particular,
if $0 \leq x^2 \leq y^2$,
\[   \Prob^x[K_\rho^2 = y^2] = \frac{\phi(x)}{\phi(y)} . \]
\end{lemma}

\begin{proof}  We may assume $x > 0$.
Since
\[     \phi''(x) = \frac{(2q-1) \, x}{x^2 + 1} \, \phi'(x), \]
the first assertion follows from It\^o's formula.  The second
assertion follows from the optional sampling theorem since
$\phi(0) = 0$.
\end{proof}

\begin{proposition}  \label{apr2.prop1}
Suppose  $K_t$ satisfies
(\ref{oct30.4.alt}) with $q > 0$.
\begin{itemize}
\item  The process is positive recurrent
with invariant density
\[      u_q(x) :=    \frac{\Gamma(q + \frac 12)}{\Gamma(\frac 12) \, \Gamma(q)}
   \, \frac{1}{(x^2+1 )^{q+\frac 12}} . \]  
In particular, there is a $c_*$ such that 
\begin{equation}  \label{nov17.21}
u_q(x) \sim c_* \, x^{-2q-1}, \;\;\;\;\;\; x \rightarrow
  \infty . 
\end{equation}
\item
\begin{equation}  \label{nov21.3}
    \int_{-\infty}^\infty \frac{x^2-1}
    {x^2 + 1}  \, u_q(dx) = \mu,
\end{equation}
where $\mu$ is as defined in \eqref{apr2.1}.
\item  There is a $c$ such that for all $l \geq 0$, $b \geq 0$,
\begin{equation}  \label{nov17.1}
 \Prob\{K_t^2 + 1 \geq b^2 \mbox{ for some } l \leq t \leq l+1
 \} \leq \frac c{(1 + b^2)^q}.
\end{equation}
\item  If  $\alpha  < q$, then
  \begin{equation}  \label{dec1.22}
    \E[
    (K_t^2 + 1)^\alpha]
  \leq    \int_0^\infty (x^2 + 1)^{\alpha }
 \,  u_q(x) \, dx < \infty.
\end{equation}
\item
\begin{equation}  \label{jan19.4}
  \Prob^x\{K_t^2 + 1 \geq r^2 \mbox{ for some } l \leq t \leq l+1
 \} \leq c \left(\frac{1 + x^2}{1+r^2}\right)^{q}  .
\end{equation}
\item  There is a $c_q$ such that for all $x$ and all $y > |x|$, if
  $\rho = \rho(y) = \inf\{t: K_t^2 = 0 \mbox{ or } K_t^2 = y^2\}$,
then
\begin{equation}  \label{jan19.3}
   \Prob^x\{K_\rho^2 = y^2\} \leq c_q \, \left(
  \frac{x^2 + 1}{y^2 + 1} \right)^q  .
\end{equation}\end{itemize}
Moreover, for all $0 < q_1 < q_2 < \infty$, $c_q$ can be chosen
uniformly over $q \in [q_1,q_2]$.
\end{proposition}

\begin{proof}
It is standard (see, e.g., \cite[p. 281]{KT})
that the stationary density
of the equation
\[   dK_t = m(K_t) \, dt + \sigma(K_t) \, dW_t \]
is a multiple of
\[      \frac{1}{\sigma^2(x)} \, \exp\left\{2 \int_0^x \frac{m(y)}{\sigma^2(y)}
  \, dy \right\} \]
(in the literature on one-dimensional diffusions a multiple
of $u_q(x) \, dx$ 
is  
called the speed measure).
The computation of $u_q$ is then a straightforward calculation (note that $q > 0$
is needed for the density to be integrable) and
\[   \int_{-\infty}^\infty \frac{u_q(x) \, dx}
    {x^2 + 1} =  \frac{\Gamma(q + \frac 12)}{\Gamma(\frac 12) \, \Gamma(q)}
         \; \frac{\Gamma(\frac 12) \, \Gamma(q+1)}{ \Gamma(q + \frac 12+1)}
 = \frac{q}{q + \frac 12}. \]
Therefore,
 \[ \int_{-\infty}^\infty \frac{x^2-1}
    {x^2 + 1}  \, u_q(dx)  = 1 - \frac{2q}{q + \frac 12}
     = \frac{1-2q}{1+2q}. \]
  \eqref{nov17.1} follows easily from \eqref{nov17.21}  and
 \eqref{nov17.3} since  the invariant measure
of  
$[b/2,\infty)$  is bounded by $c \, b^{-2q}.$.
  \eqref{jan19.3}
follows from Lemma \ref{apr2.lemma1}.

To prove \eqref{dec1.22}, note that the integral is the expectation
of $( K_t^2 +1)^{\alpha}$ assuming $K_0$ has density
$u_q$.  If instead we choose $K_0 = 0$, a simple coupling argument
shows that the expectation is decreased.  The second inequality
is immediate from the explicit form of $u_q$.
\end{proof}

As before, let
\[  L_t = \int_0^t  \frac{K_s^2-1}
   {K_s^2 + 1}  \, ds = t - \int_0^t \frac{2\,ds}{K_s^2 + 1}. \]
The relation (\ref{nov21.3}) implies that the typical value
of $L_t/t$  is $\mu$.  In fact,
we expect that $L_t = \mu t + O(t^{1/2})$. 
Lemma \ref{dec8.lemma1} quantifies this.
  We use the standard technique  
  of computing exponential moments to study the
concentration of a distribution.  The next lemma
computes the moment.

\begin{lemma}  \label{lemma.nov21.1}
 Suppose  $K_t$ satisfies
(\ref{oct30.4.alt}),  $\delta \in (-\infty,q) $,  and 
\begin{equation}  \label{nov21.1}
     p = p(\delta) = \frac{1 +2q}{4} \,\delta  - \frac{\delta^2}{4},
\end{equation}
\[ \theta = \theta(\delta) = 2p - \frac \delta 2   =
  q\delta - \frac {\delta^2}4  .\]
Then
\[     N_t = N_{t,r} :=
           e^{pL_t} \, (K_t^2 + 1)^{\delta/2} \, e^{(\theta - p)t} . \]
is a  martingale.  In particular,
\begin{equation}  \label{dec1.20}
 \E^x[e^{pL_t} \, (K_t^2 + 1)^{\delta/2}]  = (x^2 + 1)^{\delta/2} \, 
         \exp\left\{ t(\frac \delta 2 - p)\right\}. 
\end{equation}
Moreover, if  $0 \leq s < t$,
\begin{equation}  \label{dec1.21}
  \tilde \E[e^{p(L_t-L_s)} ] \leq  c \,  \exp\left\{ (t-s)  ( \frac \delta
  2 - p )
  \right\}.  
\end{equation}
 \end{lemma}

\begin{proof}
It\^o's formula shows that 
\[ d(K^2_t + 1)^{1/2} = (K_t^2 + 1)^{1/2} \, \left[\frac{K_t^2(\frac 12-q 
   ) + \frac{1}2}{K_t^2 + 1} \, dt 
  + \frac{  K_t}{\sqrt{K_t^2 +1}} \, dW_t \right]. \]
\[ d(K^2_t + 1)^{\delta/2} = (K_t^2 + 1)^{\delta/2} \, \left[\frac{K_t^2( 
  \frac{\delta^2}{2}-\delta q) + \frac{\delta}2}{K_t^2 + 1} \, dt 
  + \frac{\delta \, K_t}{\sqrt{K_t^2 +1}} \, dW_t \right]. \]
Combining this with (\ref{nov19.4}), we that
\[  d  N_t = N_t 
 \frac{\delta  \, K_t}{\sqrt{K_t^2 +1}} \, dW_t , \]
provided that
\[   \theta + \frac {\delta^2}{2} - q \delta = 0 , \;\;\;\;\;
    \frac \delta 2 + \theta - 2p = 0 . \]
Hence for all $\delta \in \R$, $ N_t$ is a  local martingale.
   If we
use Girsanov on this,  we
see that paths weighted by $ N_t$ (using stopping times
if necessary) satisfy
\[     dK_t = \left(\frac 12 - q + r  
  \right) \, K_t \, dt + \sqrt{K^2 + 1} \, d\hat W_t, \]
where $\hat W_t$ is a Brownian motion in the new measure. 
Note that this equation is of the form \eqref{oct30.4.alt} 
with a change in the parameter.
By Lemma \ref{apr2.lemma0}  we know this equation does not
have explosion in finite time and hence  
 $ N_t$
is a  martingale  and \eqref{dec1.20} follows
immediately. If $r < q$, then we can apply
Proposition \ref{apr2.prop1}. 
     To derive \eqref{dec1.21}, let $s < t$ and
let ${\cal F}_s$ denote the $\sigma$-algebra generated
by $\{K_{s'} : s' \leq s\}$.    
The martingale property implies that
\[    \E[e^{p(L_t - L_s)} \, (K_t^2 + 1)^{\delta/2}
  \mid {\cal F}_s ] =  
   \exp\left\{ (t-s)  (  \frac  \delta 2 - p)
  \right\} \, (K_s^2 + 1)^{\delta/2}  . \]
Therefore, 
\[    \E[e^{p(L_t - L_s)}  ]
    \leq \exp\left\{ (t-s)  (  \frac  \delta 2  - p)
  \right\}  \,  \E[(K_s^2 + 1)^{r/2} ]
  \leq c \, \exp\left\{ (t-s)  ( \frac \delta 2-p)
  \right\}.\]
The last inequality  holds by \eqref{dec1.22}.
\end{proof}

\begin{lemma} \label{dec8.lemma1}
 There exists $   c < \infty$ such that for all
$0 \leq s \leq  t$,
\begin{equation}  \label{nov21.5}
 \E \left[\exp\left\{  \frac{|L_t - 
   \mu    t |}{\sqrt t}\right\} \right] \leq c . 
\end{equation}
\begin{equation}  \label{nov21.5.1}
  \E \left[\exp\left\{ \frac{|(L_t-L_s) - 
  \mu(t-s)|}{ \sqrt {t-s}}\right\} \right] \leq c . 
\end{equation}
In particular, for all $\alpha >0$,
\begin{equation}  \label{nov21.6}
 \Prob\left \{   \left|L_t - 
 \mu t \right| \geq \alpha \sqrt t\right\} \leq c\,
 e^{-\alpha} ,
\end{equation}          
\begin{equation}  \label{nov21.6.1}
  \Prob\left \{   \left|(L_t-L_s) - 
\mu (t-s) \right| \geq \alpha \sqrt {t-s}\right\} \leq c\,
 e^{-\alpha} .
\end{equation} 
\end{lemma}

\begin{proof}  Recall that $ |L_t| \leq t , |\mu| \leq 1$, and hence
\[ \exp \left\{\frac{|L_t - 
\mu t | }{\sqrt t}\right\}\leq e^{2 \sqrt t}. \]
Using (\ref{nov17.1}), we can see that there is a $b$ such that
\[  \E\left[ \exp \left\{\frac{|L_t - 
 \mu t | }{\sqrt t}\right\};K_t^2 + 1 \geq e^{b \sqrt t}
  \right] \leq e^{2 \sqrt t}\,
 \Prob\{K_t^2 + 1 \geq e^{b \sqrt t}\}  
                    \leq 1. \]
Therefore, to establish (\ref{nov21.5})
it suffices to show that
\[   \E \left[\exp\left\{\frac{|L_t - 
 \mu  t  |}{\sqrt t}\right\}\, ; \, K_t^2 +1 
   \leq e^{b \sqrt t}   \right] \leq c ,\]
and to prove this it suffices to find   $c, c'$ such that
\[ \E \left[e^{L_t /\sqrt t
     }
  \, (K_t^2 +1)^{c'/\sqrt t} \right] \leq
  c\,\exp\left\{ \mu t^{1/2}  \right\}, \]
\[   \E \left[e^{  L_t /\sqrt t
     }
  \, (K_t^2 +1)^{c'/\sqrt t} \right] \leq c\,
  \exp\left\{- \mu t^{1/2}  \right\}. \]
We now use \eqref{dec1.20} with 
\[        \delta_\pm = \pm \frac{4}{1+2q} \, \frac 1 {\sqrt t}, \]
for which 
\[    p(\delta_\pm) = \pm \frac 1 {\sqrt t} + O(1/t) , \]
\[    \frac {\delta_\pm}{2} - p(\delta_\pm)
                  =  \pm \frac{1-2q}{1+2q} \, \frac{1}{\sqrt t}
    +O(1/t) =  \pm \frac{\mu}{\sqrt t} + O(1/t). \]
This gives (\ref{nov21.5}).  The estimate \eqref{nov21.5.1} is
done similarly starting with \eqref{dec1.21}.  \eqref{nov21.6}
and \eqref{nov21.6.1} follow  immediately
from the Chebyshev inequality.
\end{proof}

\begin{proposition}  \label{apr3.prop5}
For every $\delta > 0$, there exists   $c_* < \infty$
such that for all $t$
\begin{equation}  \label{apr3.10}
   \Prob\left\{\left|L_s - \mu  s\right|
   \leq c_*\,(s+2)^{\frac 12} \, \log(s+2) \mbox{ for all } 0 \leq
  s \leq t \right\} \geq 1-\delta,\end{equation}
\begin{equation} \label{apr3.11}
   \Prob\left\{\left|(L_t-L_s) -\mu (t-s)\right|
   \leq c_*\,(t-s+2)^{\frac 12} \, \log(t-s+2) \mbox{ for all } 0 \leq
  s \leq t \right\} \geq 1-\delta.
\end{equation}
\end{proposition}

\begin{proof}  Since $|\p_t L_t| \leq 1$, for every
nonnegative integer $k$ and $c_*$ sufficiently large, 
\[ \Prob\left\{\left|L_s - \mu s\right|
   \geq c_*\,(s+2)^{\frac 12} \, \log(s+2) \mbox{ for some } k \leq
  s \leq (k+1) \right\}  \]
\[  \leq   \Prob\left\{\left|L_k -\mu  k\right|
   \geq \frac {c_*}{2} \,(k+2)^{\frac 12}\, \log(k+2)  \right\}. \]

By \eqref{nov21.6}, there is a $c$ such that
\[   \Prob\left\{\left|L_k -\mu  k\right|
   \geq \frac {c_*}{2} \,(k+2)^{\frac 12}\, \log(k+2)  \right\}
   \leq  \frac{c}{(k+2)^{c_*/2}}, \]
and hence
\[   \Prob\left\{\left|L_s -\mu  s\right|
   \geq c_*\,(s+2)^{\frac 12} \, \log(s+2) \mbox{ for some } 0 \leq
  s \leq t \right\}  \leq c \, \sum_{j=0}^\infty
                       \frac{1}{(j+2)^{c_*/2}} , \]
which goes to zero as $c_*$ goes to infinity.  In particular,
we can choose $c_*$ sufficiently large so that this is
less than $\delta$.  This gives the first estimate, and the
second estimate is done similarly using
\[   \Prob\left\{\left|L_s - \mu  s\right|
   \geq c_*\,(s+2)^{\frac 12} \, \log(t-s+2) \mbox{ for some } t-k-1
 \leq s \leq t-k \right\}  \]
\[ \hspace{2in} 
  \leq  \frac{c}{(k+2)^{c_*/2}}. \]
\end{proof}

\begin{proposition}  \label{apr3.prop6}
There exists $u < \infty$ such that 
for every $\delta > 0$, there exists $  c_*, \tilde c_* < \infty$ such that for all
$t$,
\begin{equation}  \label{nye.3}
  \Prob\left\{K_s^2 + 1 
   \leq c_*\, \min\{(s+1)^u,  (t-s+1)^u\} \mbox{ for all } 0 \leq
  s \leq t \right\} \geq 1-\delta, 
\end{equation}
Moreover, on this event  we have for all $0 \leq s \leq t$, 
\[ \sigma(s) \leq \tilde c_* \, e^{2as} \, \min\{ (s+1)^u,(t-s+1)^u
\}. \]
\end{proposition}

\begin{proof}
Note that \eqref{nov17.1} implies that we
can choose $c,c_*,u$ such that 
\[    \Prob\left\{K_s^2 + 1 
   \geq c_*\, (s+1)^u \mbox{ for some } k \leq
  s \leq (k+1) \right\} \leq \frac{c}{c_* \, (k+1)^{2}}. \]
\[    \Prob\left\{K_s^2 + 1 
   \geq c_*\, (t-s+1)^u \mbox{ for some } t-(k+1) \leq
  s \leq t-k\right\} \leq \frac{c}{c_* \, (k+1)^{2}}. \]
We then derive \eqref{nye.3}    as in the previous lemma. 
\end{proof}

\subsection{Upper bound for
 Theorem \ref{dec1.theorem2}}  \label{newuppersec}

\begin{proposition}    \label{apr3.prop}
Suppose 
\begin{equation}  \label{apr3.3}
  0 \leq
  r  <  6a - 2 \sqrt{5a^2 - a}, \;\;\;\;  a \geq \frac 14 , 
\end{equation}
\begin{equation}
\label{apr3.4}   0\leq   r < 2a + \frac 12 ,\;\;\;\; a < \frac 14. 
\end{equation}
Then there exists a $c<\infty$ such that for all $x \in \R$,
\[   \E[|h_{s^2}'(x+i) |^\lambda \,(R_{s^2} + 1)^{r/2}] 
\leq \, c \,   (s+1)^{\frac{ r^2}{4a} - r} \,
  (x^2 + 1)^{r - \frac{r^2}{8a} } \,
  \log^{r - \frac{r^2}{4a}}(x^2 + 2)   , \]
In  particular, if $a > 1/4$, there exists a $c < \infty$,
such that for all $x\in \R, y \geq 2$,
 \[   \E[|h_{s^2}'(x+i) |^d] \leq \, c \,  (s+1)^{d-2} \,
 (x^2 + 1)^{1 - \frac{1}{8a}} \,
  \log^{1 - \frac{1}{4a}}(x^2 + 1) 
              . \]
\end{proposition}

     The
final assertion follows from the previous one  by plugging in 
  $r=1$  which satisfies
\eqref{apr3.3} for $a > \frac 14$. 

\begin{remark}
The upper bound in 
\eqref{dec4.10} is the special case $r=1,x=0$. 
 This proposition is a statement about $h_{s^2}$ at a fixed
time $s^2$.  In the last few sections, we saw that it is easier
to process under a particular time change.  It is not trivial
to obtain results about the original process at fixed times
by considering the time changed process.  In the proof we consider
the fixed time $e^{2at}$, and define a 
a stopping time $\tau$, which can be viewed
as a stopping time for the time-changed process, but 
for which we can prove $\tau \leq e^{2at}.$
\end{remark}

\begin{proof} 
It is easy to check that \eqref{apr3.3} and \eqref{apr3.4}
imply
\[  r  < \min\left\{  6a - 2 \sqrt{5a^2 - a},
  2a + \frac 12 \right\}. \]
Consider the martingale 
\[  M_s  = M_{s,r}(x+i)
 =  {|h_s'(x+i)|^\lambda \, Y_s^{r-\frac{r^2}{4a}}}\,
 ({R_s ^2 + 1})^{r/2}. \]
  Recall that $K_s = R_{\sigma(s)}
= e^{-as} \, X_{\sigma(s)},$ 
where as before, 
\[   \sigma(s) = \inf\{u: Y_u = e^{as}\}. \]
Let $\tau = \tau_t$ be the minimum of $t$ and
the smallest $s$ such that
\[          \sqrt{   K_s^2 + 1}
    \geq \frac{e^{a(t-s)}}{(t-s+1)} . \]
(Note that $s$ can be considered as the time in the
time-changed process, and $\sigma(s)$ is the
corresponding amount of time in the original process.)
Note that
\[  \sigma(\tau)  = \int_0^\tau e^{2as} \, [K_s^2+1]
  \, ds \leq \int_0^t \frac{e^{2at}}{(t-s+1)^2} \,ds
     \leq e^{2at}. \] 
For positive integer $k$, let $A_k= A_{k,t}$ be the event
$\{t-k < \tau \leq t- k+ 1\}$.  
Since $M_t$ is a martingale, $\tau \leq e^{2at}$, and the event
$A_k$ depends only on $M_s, 0 \leq s \leq \tau,$
 the optional sampling theorem gives
\[   \E[M_{e^{2at}} \, 1_{A_k}] = 
    \E[M_{\tau} \, 1_{A_k}]  . \]
Since $Y_t$ increases with $t$, we know that
on the event $A_k$,
\[            Y_{e^{2at}} 
 \geq Y_{\tau}  \geq   e^{at}
  \,
    e^{-ak} . \]
\begin{equation}  \label{apple.2}  
    R_\tau^2 + 1 \asymp e^{2ak} \, k^{-2} . 
\end{equation}

To compute $ \E[M_{\tau} \, 1_{A_k}] $  one only needs to
consider the time-changed process and hence we can use
the results about that process.
The Girsanov theorem,    \eqref{jan19.4}, and \eqref{apple.2}
 imply
that
\[  \E[M_{\tau} \, 1_{A_k}] = M_0\, \tilde \Prob(A_k)
  \leq  c \, (1+ x^2)^{\frac r2 + q} \, 
 e^{-2aqk} \, k^{2q} ,\]
where $q = 2a + \frac 12 - r > 0$.

Choose $k_0$ such that
\[       (k_0-1)^{-2} (e^{2a(k_0-1)} + 1 ) < x^2 + 1
   \leq k_0 ^{-2} \, (e^{2ak_0} + 1), \]
and note that
\begin{equation}  \label{apr10.2}
    k_0 \asymp \log(x^2 + 2) , \;\;\;\;
      e^{2ak_0} \asymp (x^2 + 1) \, \log^2(x^2 + 2) . 
\end{equation} 
Let
\[   V = \bigcup_{k \leq k_0} A_k. \]
Then,
\[  \E[M_{e^{2at}} \, 1_V]
  \leq \E[M_\tau \, 1_V]  \leq \E[M_0] = (x^2+1)^{r/2}. \]
On the event $V$, $Y_{e^{2at}} \geq e^{at} \, e^{-ak_0}$, hence,
if we write $s = e^{at}$, 
\begin{eqnarray*}
s^{r - \frac{r^2}{4a}}\,
   \E[|h_{s^2}'(x+i)|^{\lambda}  \, 1_V]
   & \leq & e^{ak_0 \, (r - \frac{r^2}{4a})}
    \, \E[M_{e^{2at}} \, 1_V] \\
 & \leq & e^{ak_0 \, (r - \frac{r^2}{4a})}
    \,(x^2+1)^{r/2}\\ & \leq & c \, (x^2 + 1)^{r - \frac{r^2}{8a}}
  \, \log^{r- \frac{r^2}{4a}}(x^2 + 2) .
\end{eqnarray*}
For $k > k_0 $, we use 
\begin{eqnarray*} 
s^{  r  - \frac{r^2}{4a} }\,
 \E\left[|h'_{s^2}(x+i)|^\lambda  \,  1_{A_k}\right]
  & \leq & c 
  \,
   e^{ ( r  - \frac{r^2}{4a })ak}   
 \,   \E[M_{e^{2at}} \, 1_{A_k}]\\
 &= & c 
  \,
    e^{ ( r  - \frac{r^2}{4 a})ak}   
\,\E[M_{\tau} \, 1_{A_k}]\\
  & \leq & c \, (x^2+1)^{\frac r2 + q} \,
   k^{2q}  \,    e^{( r  - \frac{r^2}{4 a}-2q) a k}.
\end{eqnarray*}
Since $ r <  6a - 2 \sqrt{5a^2 - a}$.
we get $r - \frac {r^2}{4a} - 2q < 0$, and hence
we can sum over $k> k_0$  to get
\begin{eqnarray*}
s^{r  - \frac{r^2}{4a} }\,
\E\left[|h'_{s^2}(x+i)|^\lambda  \  1_{V^c}\right]
& \leq & c \, (x^2+1)^{\frac r2 + q} \, k_0^{2q }  \,   
   e^{ ( r  - \frac{r^2}{4 a}-2q) a k_0}\\
 & \leq & c \, (x^2 + 1)^{r - \frac{r^2}{8a}} \,
   \log^{  r - \frac{r^2}{4a}}(x^2 + 2). 
\end{eqnarray*}

\end{proof}

\subsubsection{A lemma for another paper} \label{sectwo}

In \cite{LShef} we will need Lemma \ref{apple.6}
Since it   can be proved
using the   ideas in the proof of 
Proposition \ref{apr3.prop}, it is convenient to include
it here.   

\begin{lemma}  Suppose $a > 1/4$ and
\[              M_t = M_{t}(z) =
 |h_t'(z)|^d \, Y_t^{2-d} \, (R_t^2 + 1)^{\frac 12}.\] 
If
\begin{equation}  \label{apple.3}
  2a\theta \geq \max\left\{\delta, \delta  - 4a \delta + \delta^2\right\},
\end{equation}
then
\begin{equation}  \label{may1.2}
      N_t = M_t \, Y_t^{-\theta } \,
    (R_t^2 + 1)^{\frac \delta 2} 
\end{equation}
is a supermartingale.
\end{lemma}

\begin{proof}  We refer to the notations and calculations
in Proposition \ref{prop1.apr29}.   Note that
\[   N_t = |h_t'(z)|^d \, Y_t^{\frac{a(2-\theta)}a -d} 
 \, (R_t^2 + 1)^{\frac{1+\delta}{2}}. \]
   Then in the notation of that proposition,
\[ 
  2j_X =  {(1+\delta)^2} - \left(4a + 1  \right)
  \, (1 + \delta) + 2a(2-\theta) 
   =   \delta^2 + \delta - 4a \delta - 2a\theta,
\]
\[   2j_Y = 1 + \delta - 4a\left(1 + \frac 1{4a}
  \right)  + 2a(2-\theta)  = \delta -2a\theta. \]
Using \eqref{apple.3}, we see that
 $j_X,j_Y \leq 0$, and
using 
 \eqref{apr29.1} this implies
$N_t$ is a  supermartingale.
 \end{proof}

\begin{lemma}  \label{apple.6}
Suppose $a >  1/4$, $\theta \geq 0$,
 $0 \leq \delta + \theta < 2q = 4a - 1$
and $\delta,\theta$ satisfy \eqref{apple.3}.  Then
there is a $c$ such that
\[  \E\left[|h_{s^2}'(x+i)|^d \, Y_{s^2}^{2-d - \theta  }
  \, (R_{s^2}^2 + 1)^{\frac {1+\delta} 2}\right] \leq
   c\,  s^{-\theta }\,
(x^2+1)^{\frac {1 + \delta + \theta}{2}} \, 
  \log^{  \theta   }(x^2+2) . \]
\end{lemma}
 
\begin{remark}  The martingale property gives
\[  \E\left[|h_{s^2}'(x+i)|^d \, Y_{s^2}^{2-d   }
  \, (R_{s^2}^2 + 1)^{\frac {1 } 2}\right] \leq
   c\,   
(x^2+1)^{\frac {1  }{2}}  
 . \]
This proposition can be considered as a perturbation from
this.  As mentioned before, it is needed in \cite{LShef}.
\end{remark}

\begin{proof}   From the previous lemma
we know that
\[  Q_s := M_s \, (R_s^2+1)^{\delta/2} \, Y_s^{-\theta}\]
is a supermartingale.
Let $A_k,k_0,\tau,V$ be as in the proof of
Proposition \ref{apr3.prop} and write $s = e^{at}.$  The martingale
$M_s$ corresponds to $r=1$ in Proposition 
\ref{apr3.prop}.
Note that
on the event $A_k,$
\[       (R_\tau^2+1)^{\delta/2} \, Y_\tau^{-\theta}
            \asymp  e^{a(\delta+\theta) k} \, k^{-\delta}\,
              s^{- \theta}. \] 
Since $Q$ is a supermartingale and $\tau \leq s^2$,
\[ 
  \E\left[Q_{s^2} \, 1_{A_k}\right]
  \le \E\left[Q_{\tau} \, 1_{A_k} \right] 
   \leq  c \,  e^{a(\delta + \theta) k} \, k^{-\delta}\,
              s^{-  \theta} \, \E\left[M_\tau \, 1_{A_k} \right].\]

This implies
\begin{eqnarray*}
 s^{\theta} \,
  \E\left[Q_{s^2} \, 1_{V}\right] & \leq &
 c  \,  \max\{k^{-\delta} \,
   e^{a(\delta + \theta) k}: k=1,\ldots k_0\}
 \, \E[M_\tau \, 1_V] \\ & \leq & c 
 (x^2 + 1)^{\frac 12 } \, \max\{k^{-\delta} \,
   e^{a(\delta + \theta) k}: k=1,\ldots k_0\} \\
    & \leq & c \,   (x^2 + 1)^{\frac 12}
   \, k_0^{-\delta} \, e^{a(\delta + \theta)k_0}.
\end{eqnarray*}
Recalling that $k_0^{-2} \, e^{2a k_0} \asymp (x^2+1)$, we get
\[ s^{\theta} \,
  \E\left[Q_{s^2} \, 1_{V}\right]
  \leq c \, (x^2 + 1)^{\frac{1+ \delta+\theta}{2}} \,
   \log^{ \theta }(x^2 + 2) . \]  

Also,
\begin{eqnarray*}
s^{  \theta}  
 \, \E\left[Q_{s^2} \, 1_{V^c}\right]
   & \leq & c \sum_{k > k_0}   e^{a(\delta+\theta) k} \, k^{-\delta}\,
           \E[M_\tau \, 1_{A_k}] \\
   & \leq & c \sum_{k > k_0}  e^{a(\delta+\theta) k} \, k^{-\delta}\,
      (x^2+1)^{q + \frac 12} \, e^{-2kqa } \, k^{2q} \\
  & \leq & c \,  e^{a(\delta+\theta-2q) k_0} \, k_0^{2q-\delta}\,
      (x^2+1)^{q + \frac 12} \\
 &\leq & c \,  (x^2 + 1)^{\frac{1+ \delta+\theta}{2}} \,
   \log^{ {\theta} }(x^2 + 2)  . 
\end{eqnarray*}
The penultimate
 inequality uses the fact that $\delta + \theta < 2q$.
\end{proof}

\subsubsection{Some corollaries}  \label{somecorsec}

The next corollaries are useful in determining the existence
of the SLE curve.

\begin{corollary}  \label{exist1}
For every $a > 1/4$,  there exist
$C,\lambda,\zeta$ such that $\lambda + \zeta > 2$ and 
for all $x$,
\[   \E\left[|h_{t^2}'(x+i)|^\lambda\right]
   \leq c \, (x^2 + 1)^{2a } \, (t+1)^{-\zeta} . \]
\end{corollary}

\begin{proof}  If we choose $\epsilon$ suffciently small, then
$r = 1 + \epsilon$ satisfies \eqref{apr3.3}, and
\[     \left[ r  \left(1 + \frac{1}{2a}\right) -
  \frac{r^2}{4a}\right]    + \left[r - \frac{r^2}{4a}\right] > 2. \] 
\end{proof}

\begin{corollary}  \label{exist2}
For every $0 < a < 1/4$ and every
$\delta > 0$,  there exist $C,\lambda,\zeta$
such that $\lambda + \zeta > 2$ and 
for all $x$,
\[   \E\left[|h_{t^2}'(x+i)|^\lambda\right]
   \leq c \, (x^2 + 1)^{a + \frac 14 + \delta } \, (t+1)^{-\zeta} . \]
\end{corollary}

\begin{proof} We consider $r = 2a + \frac 12 - \epsilon.$
\end{proof}

\subsection{Lower bound for
 Theorem \ref{dec1.theorem2}  } \label{returnsec}

\begin{proposition} \label{general2} Let $a > 0$ and  $ 0 <
r < 2a + \frac 12$.  Let
\[  \beta = \frac \mu 2
= \frac{1-2q}{2 + 4q} = \frac{r - 2a}
  {2 + 4a - 2r}.\]
 There exist
$c >0 $  and a subpower function $\phi_0$    such
that the following holds.   
Let $ X_t=X_t(i), Y_t=Y_t(i)$,
and let $E = E(\phi_0,t)$ be the event that for all $1 \leq s
\leq t$, 
\[               {\sqrt s}\, \max\left\{\frac{1}{\phi(s)} ,
  \frac{1}{\phi(t/s)}\right\} \leq 
    Y_s  \leq \sqrt{2as +1}
      , \]
\[     \frac{s^\beta}{\phi(s)} \leq |h_s'(i)|
    \leq  s^\beta \, \phi(s)
   , \;\;\;\;\;
       \frac{(t/s)^\beta}{\phi(t/s)} \leq 
  \frac{|h_t'(i)|}{|h_s'(i)|}
    \leq  (t/s)^\beta \, \phi(t/s)
   , \]
\[  |X_s| \leq   \sqrt s \, \min\left\{\phi(s),\phi(t/s)\right\}\
  . \]
Then,  
for all $t \geq 0$, 
\[ 
 \E\left[\left|h_t'(i)\right|^d \, 1_E \right]   
  \geq c \, t^{\frac{r^2}{4a} - r}. \] 
\end{proposition}

\begin{remark} For     $r=1, a > 1/4$,
\[               \beta = 2 \mu = \frac{1-2a}{4a}, \]
which agrees with our previous definition.
\end{remark}

\begin{remark}  The idea of the proof is simple.  We have
already shown that for the time-changed process, a certain
event has positive probability with respect to the weighted
measure.  In this event, $K_s$ does not get too large, and from
this we can show that the amount of time to traverse the paths
in the original parametrization is about what we would expect
it to be.
\end{remark}

\begin{proof}

Let $M_t = |h_t'(i)|^\lambda Y^{r -\frac{r^2}{ra}}_t
  \, (R_t^2+1)^{r/2}$ be the martingale.  As before, let
\[   \sigma(t) = \inf\{s: Y_s = e^{at} \}. \]
If $V$ is an event depending only on $\{B_s: s \leq \sigma(t)\}$,
then the Girsanov theorem tells us that
$\E[M_{\sigma(t)} \, 1_V]$ is the probability of $V$ under
the appropriately
weighted measure.  In our case, the paths under the weighted
measure satisfy
the SDE from Section \ref{SDEsec}.  We can therefore use
Propositions \ref{apr3.prop5}
and  \ref{apr3.prop6} to say that there exist positive
constants $c_*,u$ and an event $V= V_t$ such that
\[   \E\left[M_{\sigma(t)} \, 1_V\right] \geq \frac 12 , \]
and such that on $V$, for $0 \leq s \leq t$,
\[          |L_s - \mu s| \leq c_* \, (s+2)^{\frac 12}
  \, \log(s+2), \]
\[   |(L_t - L_s) - \mu(t-s)| \leq c_* \, (t-s+2)^{\frac
  12} \, \log(t-s+2) , \]
\[    K_s^2 + 1 \leq c_* \, \min \left\{(s+1)^u,(t-s+1)^u\right\} . \]
Here $K_s = R_{\sigma(s)}$ and
$        |h_{\sigma(s)}'(i)| = e^{a L_s} .$  Roughly speaking, we
would like to say on the event $V$,
\[    |h_{e^{2as}}'(i)| \approx |h_{\sigma(s)} '(i) |
    \approx e^{a \mu s} = [e^{2as}]^\beta. \] The definition
of $V$ justifies the second relation and the  
equality holds by definition.  We need to justify
the first relation.

We claim that there exist $c_1,c_2$ such that on the event
$V$,
\[   c_1 \, [e^{2as}-1] \leq \sigma(s) \leq c_2 \, e^{2as}
  \, \min\left\{(s+1)^u,(t-s+1)^u \right\}, \;\;\;\;
  0 \leq s \leq t. \]
We have already noted that the first inequality holds for all
paths.  To derive the second, 
\[  \sigma(s) =\int_0^s e^{2av} \, (K_v^2 + 1) \, dv 
\leq c_* \, \int_0^s e^{2ar} (v+1)^u \, dr
\leq c_* \, (s+1)^u \, e^{2as}, \]
\begin{eqnarray*} 
  \sigma(s) \leq c_* \, \int_0^s e^{2av} (t-r+1)^u \, dv
      & \leq  &   c_* \, e^{2as}\, (t-s+1)^u  \int_0^s e^{2a(v-s)} \,\left(
   s-r+1\right)^u \, dv \\
   & \leq &   c_* \, e^{2as}\, (t-s+1)^u \, \int_0^\infty e^{-2ay} \, (y+1)^u
  \, dy \\
 & \leq & \tilde c_* \, e^{2as}\, (t-s+1)^u.
\end{eqnarray*}
In particular,
\[        \tilde c_1 e^{2at} \leq  \sigma(t) \leq c_2 \, e^{2at}. \]   
 
By inverting, we see that there exist $c_3,c_4$ such that on
the event $V$
\[               c_3 \, \max\left\{(s+1)^{-u},
    (t-s+1)^{-u} \right\} \, e^{s} \leq Y_{e^{2as}}  \leq
   c_4  \, e^{s} , \]
at least for $e^{as} \geq c_3 \, e^{2at}. $   
By using properties of the Loewner equation, we see that there is
a $c_5$ such that for all $\sigma(t) \leq s \leq c_2 \, e^{2at},$
\begin{equation}
\label{apr4.1}
    c_5^{-1} \, |h_s'(i)| \leq |h_{\sigma(t)}'(i)|
  \leq c_5 \, |h_s'(i)|. 
\end{equation}
Let $\tilde V$ be the intersection of the event $V$ with the event
\[            \sup\{|B_s - B_{\sigma(t)}| \leq
          e^{at}, \sigma(t) \leq s \leq  c_2 \, e^{2at}\}. \]
By \eqref{apr4.1} 
and standard properties  of
Brownian motion
\[   \E\left[  |h_{c_2 \, e^{2at}}'(i)|^\lambda \, 1_{\tilde V}\right]
  \geq c \,    \E\left[|h_{\sigma(t)}'(i)|^\lambda \, 1_{\tilde V} \right]
               \geq  c \, \E\left[|h_{\sigma(t)}'(i)|^\lambda \, 1_V\right]
.\]
One can check that the event $\tilde V$ satisfies the necessary
properties.  We have established the result for the time $c_2 e^{2at}$
(but every time can be written this way for some $t$).
\end{proof}

\section{Upper bound}  \label{uppersec}

For the sake of completeness, we sketch
a proof of the
upper bound for the Hausdorff dimension using a version
of the argument from \cite{RS}.  It suffices to show that
for every $r,t < \infty$ there is a $c$ such that for
every $\epsilon > 0$ and every $z \in \rect(r) = [-r,r] \times
[1/r,r]$,
\begin{equation}  \label{dec4.20}
  \Prob\{\dist(z,\gamma[0,t]) \leq \epsilon \}
  \leq c \, \epsilon^{2-d}. 
\end{equation}
Indeed, if this holds, then the expected number of balls of
radius $\epsilon$ need to cover $\gamma[0,t]
\cap \rect(r)$ is $O(\epsilon^{-d})$, 
and from this it is easy to conclude
that with probability one $\hdim[\gamma[0,\infty)] \leq d$.
The Koebe-$(1/4)$ Theorem can be used to show that
$ \dist[z,\gamma[0,t] \cup \R]$ is comparable to
$\distsub_t:=Y_t/|g_t'(z)|$.  Hence \eqref{dec4.20} follows
from the following proposition.

\begin{proposition}
For every $z = y(x+i)   \in \Half$,  
\begin{equation}  \label{apple.7}
  \Prob\{\distsub_\infty \leq \epsilon\} \sim c_* \, G(z)
  \, \epsilon^{2-d},\;\;\;\;\;\; \epsilon \rightarrow 0+,
\end{equation}
where 
\[           G(y(x+i)) = y^{d-2} \, (x^2 + 1)^{\frac 12 -2a} , \]
and
\[            c_* =2 \, \left[\int_0^\pi \sin^{4a } \theta
  \, d \theta \right]^{-1} . \]
Moreover, the rate of convergence is uniform on every
compact $K \subset \Half$.
\end{proposition}

\begin{remark} The estimate \eqref{apple.7} with $\approx$ replacing
$\sim$ was first established in \cite{RS}.  The weaker estimate
is sufficient for proving the upper bound on Hausdorff dimension.
In the proof of the lower bound in  \cite{Beffara}, the
$\approx$ was replaced with $\asymp$.
There is a proof of \eqref{apple.7} in \cite{LBook} which
follows more closely that proof of \cite{RS} but uses asymptotics
for complex hypergeometric functions and is not very intuitive.
The proof below which uses Girsanov is more natural.  The strong
asymptotics are used in \cite{LShef} to motivate the definition
of the natural parametrization, so it seems useful to give
a simple proof.
\end{remark}

\begin{proof}
Consider
the usual (forward) $SLE$
\[          \dot g_t(z) = \frac{a}{g_t(z) - U_t }, \;\;\;\;
    g_0(z) = z , \]
with $U_t = -B_t$.  For $z \in \Half$, let $Z_t =
Z_t(z) = X_t + i Y_t = g_t(z) - U_t, \Theta_t=
\Theta_t(z) = \arg(Z_t(z))$ and note that
\[   dX_t = \frac{a\, X_t}{X_t^2 + Y_t^2} \, dt
   + dB_t, \;\;\;\;\;  \p_t Y_t = -\frac{aY_t}{X_t^2 + Y_t^2} . \]
\[   \p_t \distsub_t =  -\distsub_t\, \frac{2a \, Y_t^2}
   {(X_t^2 + Y_t^2)^2}, \]
\[   d \Theta_t = \frac{(1-2a) \, X_t \, Y_t}{(X_t^2 + Y_t^2)^2}
 \, dt - \frac{Y_t}{X_t^2 + Y_t^2} \, dB_t. \]
Combining all of this we can show that
\[   M_t:=  \distsub_t^{d-2} \, \sin^{4a-1} \, \Theta_t \]
is a local martingale with $M_0 = G(z)$.  We can use
the Girsanov theorem (using the stopping time $\tau_\epsilon
= \inf\{t: \distsub_t \leq \epsilon\}$) to weight the paths by
the local martingale $M_t$.  Then,
\[   d \Theta_t = \frac{2a \, X_t \, Y_t}{(X_t^2 + Y_t^2)^2}
 \, dt - \frac{Y_t}{X_t^2 + Y_t^2} \, dW_t, \]
where $W_t$ is a standard Brownian motion in the new measure.

   Since $\distsub_t$ decreases in $t$,
we can do a time change so that $\distsub_t$ decays 
deterministically.  If we choose $\sigma(t)$ so that $\hat
\distsub_t := \distsub_{\sigma(t)} = e^{-2at}, $ then
$\hat \Theta_t = \Theta_{\sigma(t)}$  satisfies
  \begin{equation}
\label{jun4.1}    d\hat \Theta_t = 2a  \, \cot \hat \Theta_t
  \, dt + d \hat  W_t ,  \end{equation}
where $\hat W_t$ is a standard Brownian motion (in the weighted measure).
If $\hat \E$ denotes expectations with respect to the new measure,
then
\[
  e^{-2at(d-2)} \, \Prob\{\distsub_\infty \leq e^{-2at}\}
  = \E[\hat M_t \, \sin^{1-4a} \hat \Theta_t; \distsub_t \leq e^{-2at}]
    = M_0 \,
 \hat \E[ \sin^{1-4a} \hat \Theta_t] . 
\]
Therefore,
\[  
  \Prob\{\distsub_\infty \leq e^{-2at}\} = G(z) \,  e^{-2at(2-d)} 
\,  e(t,\arg(z)), \]
where $e(t,\theta) = \hat \E [\sin^{1-4a} \hat \Theta_t\mid
\hat \Theta_0 = \theta]$ where
$\hat \Theta_t$ satisfies \eqref{jun4.1}.  One can check that
the invariant density for \eqref{jun4.1} is $f(\theta) = c \,
\sin^{4a} \theta$, and hence 
\[    e(\infty,\theta) = \frac{c\int_0^\pi \, \sin \theta \, d\theta}
           {c \int_0^\pi \sin^{4a} \theta \, d\theta }
= c_*. \]
The final statement about uniform convergence concerns the rate
of convergence to the invariant distribution.  We leave the simple
argument to the reader.
\end{proof}

\section{Continuity in capacity parametrization}  \label{curvesec}
 
Here we consider solutions of
\begin{equation}  \label{jan4.1}
            \dot g_t(z) = \frac{a}{g_t(z) - U_t}, \;\;\;
  g_0(z) = z, 
\end{equation}
where $a > 0$ is fixed and $U_t$ is continuous.  As before,
let $f_t = g_t^{-1}$ and for $y > 0$, let
\[          v(t,y) = \int_0^y |f_t'(U_t + ir)| \, dr. \]
We say that $t$ is an {\em accessible} time if  the limit
\[             \gamma(t) = \lim_{y \rightarrow 0+}
                  f_t(U_t + i y)  \]
exists.  If 
$v(t,y) < \infty$ for
 some $y > 0$ then  $v(t,0+) = 0$, $t$ is an
accessible time,
and 
 \begin{equation}  \label{jan4.5}
        |\gamma(t) - f_t(U_t + iy)| \leq v(t,y) . 
\end{equation}
To show that $\gamma$ is continuous in the capacity parametrization
(i.e., that $\gamma$ is a curve) one first has to show that each $t$
is accessible (so that $\gamma$ is well defined) and then to
show continuity.  The
strategy to show continuity is   
  a ``$4 \epsilon$''-argument,
\[ |\gamma(s) - \gamma(t)| \leq
 |\gamma(s) - f_s(U_s + iy)| + |\gamma(t) - f_t(U_t+iy)| \hspace{1in} \]
\[ \hspace{1in}
   + |f_s(U_s+iy) - f_s(U_t + iy)| + |f_s(U_t + iy)  -f_t(U_t+iy)|. \]

\begin{lemma}
Suppose that $g_t$ satisfies \eqref{jan4.1}; $y >0$;
  $s<t$ are 
accessible times with $t-s \leq y^2/(4a)$ and
$\max_{s \leq r \leq t}|U_s - U_r| \leq y/4$.  Then
\begin{equation}  \label{jan4.2}
|f_t'(U_t + iy)| \leq  \frac 9 2 \, e^{1/4} \,  |f_s'(U_s + iy)|
\end{equation}
\begin{equation}  \label{jan4.3}
    |f_t(U_t + iy) - f_s(U_s + iy)| \leq  8 v(s,y)
  ,
\end{equation}
\begin{equation}  \label{jan4.4}
|\gamma(t) - \gamma(s)| \leq    25 \,[v(s,y) + v(t,y)].
\end{equation}

\end{lemma}

\begin{proof}
	   Recall
from Section \ref{nov15sec} that
\[            f_t(U_t + iy) = f_s(z_0) , \]
where $z_0 = F_{t-s}(U_t + iy)$ and $F_r$ is the solution to
the  time-reversed Loewner equation
\[                  \p_r F_r(w) = \frac{a}
                         {U_{t-r} - F_r(w)}, \;\;\;
                F_0(w) = w . \]
The imaginary part increases
so we get the bound $|\p_r F_r(x+iy)| \leq a /y$ which implies
$|F_{t-s}(U_t+iy) - (U_t+iy)| \leq y/4$ and hence
 \begin{equation}  \label{jan4.6}
                 |z_0 - (U_s + iy)|
  \leq |z_0 - (U_t + iy)| + |U_t - U_s| \leq y/2. 
\end{equation}
 Similarly, by differentiating
the equation we get the bound
\[            e^{-1/4} \leq    |F_{t-s}'(x+iy)|  \leq e^{1/4}. \]
The distortion theorem   implies
$
 |f_s'(z_0)| \geq (2/9) \, |f_s'(U_s + iy)| , $
and hence
\begin{equation}  \label{jan4.8}
    |f_t'(U_t + iy)| \geq  \frac 29\,e^{-1/4} \, |f'_s(U_s + iy)| .
\end{equation}
The distortion theorem and \eqref{jan4.6}  give 
\begin{eqnarray*}
|f_s (z_0) - f_s(U_s + iy)| & \leq & 2   \, |f'_s(U_s + iy)|\\
    & \leq & 9 \,e^{1/4} \min\{|f'_s(U_s + iy)|, |f_t'(U_t + iy) |\}\\
    & \leq & 36\, e^{1/4} \min\{v(s,y), v(t,y)\}
\end{eqnarray*}
The estimate \eqref{jan4.3} follows from the first inequality
and $|f'_s(U_s + iy)| \leq 4 \, v(s,y)$, and 
 \eqref{jan4.4} follows from  the final inequality,
\eqref{jan4.5}, and the estimate $36 e^{1/4} + 2 < 50$.
\end{proof}

\begin{lemma}  There exists a $C_0$ such that if $g_t$
satisfies \eqref{jan4.1}, $R \geq 0$, $t \leq R^2/a$ and
$|U_s- U_0| \leq R$, $0 \leq s \leq t$, then
$\Half \setminus H_t$ is contained in the ball of 
radius $C_0 R$ about the origin.
\end{lemma}

\begin{proof}
We leave this to the reader; see \cite[Section 3.4]{LBook}
if one wants a proof.
\end{proof}

In particular, if $t$ is an accessible point,
\[         |\gamma(t) - U_0| \leq C \, \max\left\{\sqrt t,
       \max_{0 \leq s \leq t} |U_s - U_0| \right\}. \]
We say that $g_t$ is generated by a curve if every $t$ is an
accessible time and $\gamma(t)$ is a continuous function of $t$.
The last estimate shows that $\gamma$ must be right continuous
at $0$.  Hence we get the following.

\begin{corollary}  Suppose $g_t$ is a solution to \eqref{jan4.1}.
Let
\[    v_\delta(y) =  \sup_{\delta \leq t \leq 1/\delta}
           v(t,y) = 0 . \]
Suppose that for every $\delta > 0$,
$ v_{\delta}(0+)  = 0 .$
Then $g_t$ is generated by a curve $\gamma$.  Moreover, if $\delta
\leq s \leq t \leq s +(y^2/4a) \leq 1/\delta$ and $\max_{s \leq
r \leq t}|U_s - U_r|
\leq y/(4a)$,  $|\gamma(s) - \gamma(t)|  \leq 50 v_\delta(y)$.
\end{corollary}

The following is essentially a restatement of the corollaries
in Section \ref{somecorsec}.  We only need the lemmas
for $x=0$. 

\begin{lemma} If $a \neq 1/4$, there exist 
$c < \infty$,$\lambda,\zeta >0$
with $ \lambda+\zeta > 2$  
such that for all $x,t$,
\[  \E\left[|h_{t^2}'( i)|^\lambda\right] \leq c   \, (t+1)^{-\zeta}. \]
\end{lemma}

Let $\lambda,\zeta$ be as in the lemma and choose
 $\theta < 1$  with $\zeta + \theta \lambda > 2$.
\[ 
\Prob\{|h_{t^2}'(iy)| \geq y^{-\theta}\}  \leq  
          y^{\theta \, \lambda} \, 
\E\left[|h_{t^2}'(iy)|^\lambda\right] 
=   y^{\theta \, \lambda} \, \E\left[|h_{t^2/y^2}'(i)|^\lambda\right] 
 \leq c \, (t+1)^{-\zeta} \, y^{\theta \lambda + \zeta}.\]

\begin{corollary}  For every $\kappa \neq 8$, there
exists a $\theta < 1$ such that the following holds.
   Let $g_t$ denote the conformal maps of $SLE_\kappa$
and $f_t = g_t^{-1}$. Let
\[       \Theta_K = \Theta_K(\theta )=
  \sup\left\{ y^{-\theta} \, |f_t'(U_t + iy)|:
      \frac 1K \leq t \leq K, \;\;\;
   0 < y \leq 1 \right\} < \infty . \]
Then
  with probability one, for every $K < \infty$,
$\Theta_K < \infty$.  In particular, the path is generated
by a curve $\gamma$, and
\[        |\gamma(t) - f_t(U_t + iy)| \leq
          \frac{\Theta_K}{1-\theta} \, y^{1-\theta},
\;\;\;\;\;   \frac 1K \leq t \leq K, \;\;\;
              0 < y \leq 1. \]
Moreover, if $\epsilon < (1-\theta)/2$,
there is a $C = C(\omega,\theta,K)$ such that
for all $s,t \leq K$,
\[   |\gamma(s) - \gamma(t) | \leq C \, |s-t|^{\epsilon}. \]
\end{corollary}

\begin{proof}  It suffices to prove the result
for each positive integer $K$ and let
$    D_{n} = \{k 2^{-n} : k =1,2,\ldots\} . $
 Let us fix such
a $K$. We choose $r$  and $\zeta,\lambda$
as in Corollaries \ref{exist1} and
\ref{exist2}  
and  choose $\theta<1$ satisfying $\zeta +\lambda\theta > 2$.
   Then the proposition combined
with the Borel-Cantelli Lemma shows that  
with probability one
\[    \sup \left\{y^{-\theta} \, |f_t'(U_t + iy)|:
       : y  \in D_n \cap [0,1], t \in D_{[2n(1+\epsilon)]}
  \cap [0,K],
   n=1,2,\ldots  \right\}  < \infty. \]
Also, it is known that with probability one.
\[              \sup\left\{\frac{|U_t-U_s|}{|t-s|^{\frac 12 - \epsilon}}
   \;: \;0 \leq t,s \leq K \right\} < \infty.\]
We now use \eqref{jan4.8}.
\end{proof}


\begin{thebibliography}{00}

\bibitem{Beffara} V. Beffara, The dimension of the SLE curves,
preprint, to appear in Annals of Probab.

\bibitem{Falconer} K. Falconer (1990). {\em Fractal
Geometry: Mathematical Foundations and Applications}.
John Wiley \& Sons.

\bibitem{Kam}  N,-G. Kang (2007).  Boundary behavior
of SLE, Journal of AMS,  {\bf 20}, 185--210.

\bibitem{KT} S. Karlin and H. Taylor (1981). {\em A Second
Course in Stochastic Processes}, Academic Press.


\bibitem{Kennedy} T. Kennedy (2005).
Monte Carlo comparisons of the self-avoiding walk 
and SLE as parameterized curves, preprint.

 

\bibitem{LBook} G. Lawler (2005). {\em Conformally Invariant
Processes in the Plane}, Amer. Math. Soc.
 

\bibitem{LSWlerw} G. Lawler. O. Schramm, and W. Werner
(2004).  Conformal invariance of planar loop-erased
random walks and uniform spanning trees, Annals of
Probab. {\bf 32}, 939--995.

\bibitem{LSWrest} G. Lawler, O. Schramm. and W. Werner
(2003).  Conformal restriction: the chordal case.
J. Amer. Math. Soc. {\bf 16}, 917--955.
 

\bibitem{LSWsaw}  G. Lawler. O. Schramm, and W. Werner
(2004). On the scaling limit of planar self-avoiding
walk, in {\em Fractal Geometry and Applications: A Jubilee
of Benoit Mandelbrot}, Vol. II, M. Lapidus, m. van
Frankenhuijsen, ed., Amer. Math. Soc. 

\bibitem{LShef} G. Lawler and S. Sheffield, Construction
of the natural parametrization for $SLE$ curves, in preparation.

\bibitem{RS} S. Rohde and O. Schramm (2005). Basic properties of
SLE, Annals of Math. {\bf 161}, 879--920.

\bibitem{Schramm} O. Schramm (2000). 
Scaling limits of loop-erased random walks
and uniform spanning trees, Israel J. Math. {\bf 118}, 221--288.

\bibitem{SSharm}  O. Schramm and S. Sheffield (2005).
Harmonic explorer and its
convergence to SLE(4), Annals of Probab. {\bf 33}, 2127--2148.



\end{thebibliography}
\end{document}